\documentclass[12pt]{amsart}

\usepackage{amssymb}
\usepackage{stmaryrd}
\usepackage{amscd}
\usepackage{calc}
\usepackage{ifthen}

\input prepictex   \input pictex    \input postpictex

\newcounter{boxsize}
\newcounter{tempcounter}
\newcounter{tempdist}
\newcounter{tempe}
\newcommand{\smallentryformat}{\scriptstyle\sf}
\def\arr#1#2{\arrow <2mm> [0.25,0.75] from #1 to #2}
\def\num#1{$\scriptscriptstyle\sf#1$}
\def\smallsq#1{\plot 0 0  0.#1 0  0.#1 0.#1  0 0.#1  0 0 /}
\def\ssq{$\smallsq2$}
\newcommand\smbox{\put(0,0){\line(1,0){\value{boxsize}}}%
  \put(\value{boxsize},0){\line(0,1){\value{boxsize}}}%
  \put(0,0){\line(0,1){\value{boxsize}}}%
  \put(0,\value{boxsize}){\line(1,0){\value{boxsize}}}}
\newcommand\numbox[1]{\put(0,0)\smbox%
  \put(0,0){\makebox(\value{boxsize},\value{boxsize})[c]{%
      $\smallentryformat#1$}}}

\newcommand\singlebox[1]{\raisebox{-.4ex}{\begin{picture}(4,0)\setcounter{boxsize}{4}%
    \put(0,0)\smbox%
    \put(0,0){\makebox(\value{boxsize},\value{boxsize})[c]{%
      $\scriptstyle\sf#1$}}\end{picture}}}

\newcommand\vdotbox{\setcounter{tempcounter}{\value{boxsize}*2}
  \multiput(0,-\value{boxsize})(\value{boxsize},0)2{%
    \line(0,1){\value{tempcounter}}}
  \put(0,\value{boxsize}){\line(1,0){\value{boxsize}}}
  \put(0,-2){\makebox(\value{boxsize},\value{tempcounter})[c]{%
      $\scriptscriptstyle\vdots$}}}
\newcommand\vdotboxI{\setcounter{tempcounter}{\value{boxsize}*2}
  \put(0,-\value{boxsize}){\makebox(\value{boxsize},\value{tempcounter})[c]{%
      $\scriptscriptstyle\vdots$}}}
\newcommand\rectbox[1]{\setcounter{tempcounter}{#1*\value{boxsize}}
  \put(0,0){\line(1,0){\value{boxsize}}}
  \put(0,\value{tempcounter}){\line(1,0){\value{boxsize}}}
  \put(0,0){\line(0,1){\value{tempcounter}}}
  \put(\value{boxsize},0){\line(0,1){\value{tempcounter}}}}
\newcommand\rectboxL[1]{\setcounter{tempcounter}{#1*\value{boxsize}}
  \put(0,0){\line(1,0){\value{boxsize}}}
  \put(0,\value{tempcounter}){\line(1,0){\value{boxsize}}}
  \put(0,0){\line(0,1){\value{tempcounter}}}
  \setcounter{tempcounter}{2*\value{boxsize}}
  \put(\value{boxsize},0){\line(0,1){\value{tempcounter}}}}
\newcommand\rectboxJ[1]{\setcounter{tempcounter}{#1*\value{boxsize}}
  \put(0,0){\line(1,0){\value{boxsize}}}
  \put(0,\value{tempcounter}){\line(1,0){\value{boxsize}}}
  \put(\value{boxsize},0){\line(0,1){\value{tempcounter}}}}
\newcommand\boxes[2]{\setcounter{tempcounter}{#1*\value{boxsize}/2}
  \multiput(0,-\value{tempcounter})(0,\value{boxsize}){#1}\smbox
  \ifthenelse{#2=1}{\put(0,-\value{tempcounter}){%
      \makebox(\value{boxsize},\value{boxsize})[c]{$\smallentryformat1$}}}{}%
  \ifthenelse{#2=2}{\put(0,-\value{tempcounter}){%
      \makebox(\value{boxsize},\value{boxsize})[c]{$\smallentryformat2$}}}{}%
  \ifthenelse{#2=3}{\put(0,-\value{tempcounter}){%
      \makebox(\value{boxsize},\value{boxsize})[c]{$\smallentryformat2$}}%
    \setcounter{tempcounter}{\value{tempcounter}-\value{boxsize}}%
  \put(0,-\value{tempcounter}){\makebox(\value{boxsize},
    \value{boxsize})[c]{$\smallentryformat1$}}}{}}
\newcommand\rrboxes[2]{\setcounter{tempcounter}{#1*\value{boxsize}/2}
  \multiput(0,-\value{tempcounter})(0,\value{boxsize}){#1}\smbox
  \ifthenelse{#2=1}{\put(0,-\value{tempcounter}){%
      \makebox(\value{boxsize},\value{boxsize})[c]{$\smallentryformat3$}}}{}%
  \ifthenelse{#2=2}{\put(0,-\value{tempcounter}){%
      \makebox(\value{boxsize},\value{boxsize})[c]{$\smallentryformat4$}}}{}%
  \ifthenelse{#2=3}{\put(0,-\value{tempcounter}){%
      \makebox(\value{boxsize},\value{boxsize})[c]{$\smallentryformat4$}}%
    \setcounter{tempcounter}{\value{tempcounter}-\value{boxsize}}%
  \put(0,-\value{tempcounter}){\makebox(\value{boxsize},
    \value{boxsize})[c]{$\smallentryformat3$}}}{}
  \setcounter{tempcounter}{\value{tempcounter}-\value{boxsize}}
  \put(0,-\value{tempcounter}){\makebox(\value{boxsize},
    \value{boxsize})[c]{$\smallentryformat2$}}
  \setcounter{tempcounter}{\value{tempcounter}-\value{boxsize}}
  \put(0,-\value{tempcounter}){\makebox(\value{boxsize},
    \value{boxsize})[c]{$\smallentryformat1$}}
}

\newcommand\bboxes[2]{\setcounter{tempcounter}{#1*\value{boxsize}/2}
  \setcounter{tempdist}{-\value{tempcounter}+#2*\value{boxsize}-\value{boxsize}}
  \multiput(0,-\value{tempcounter})(0,\value{boxsize}){#1}\smbox
  \ifthenelse{#2=0}{}{\put(0,\value{tempdist}){%
      \makebox(\value{boxsize},\value{boxsize})[c]{$\bullet$}}}}

\newcommand\lboxes[2]{\setcounter{tempcounter}{#1*\value{boxsize}/2}
  \setcounter{tempdist}{-\value{boxsize}/2}
  \setcounter{tempe}{-\value{tempcounter}+#1*\value{boxsize}-#2*\value{boxsize}}
  \multiput(\value{tempdist},-\value{tempcounter})(0,\value{boxsize}){#1}\smbox
  \put(\value{tempdist},-\value{tempcounter}){%
    \makebox(\value{boxsize},\value{boxsize})[c]{$\smallentryformat 2$}}
  \setcounter{tempdist}{\value{tempdist}+\value{boxsize}}
  \multiput(\value{tempdist},\value{tempe})(0,\value{boxsize}){#2}\smbox
  \put(\value{tempdist},\value{tempe}){%
    \makebox(\value{boxsize},\value{boxsize})[c]{$\smallentryformat 1$}}}
\newcommand\pboxes[2]{\setcounter{tempcounter}{#1*\value{boxsize}/2}
  \setcounter{tempdist}{-\value{boxsize}/2}
  \setcounter{tempe}{-\value{tempcounter}+#1*\value{boxsize}-#2*\value{boxsize}}
  \multiput(\value{tempdist},-\value{tempcounter})(0,\value{boxsize}){#1}\smbox
  \put(\value{tempdist},-\value{tempcounter}){%
    \makebox(\value{boxsize},\value{boxsize})[c]{$\smallentryformat 2_{#2}$}}
  \setcounter{tempdist}{\value{tempdist}+\value{boxsize}}
  \multiput(\value{tempdist},\value{tempe})(0,\value{boxsize}){#2}\smbox
  \put(\value{tempdist},\value{tempe}){%
    \makebox(\value{boxsize},\value{boxsize})[c]{$\smallentryformat 1$}}}

\newcommand\rrlboxes[2]{\setcounter{tempcounter}{#1*\value{boxsize}/2}
  \setcounter{tempdist}{-\value{boxsize}/2}
  \setcounter{tempe}{-\value{tempcounter}+#1*\value{boxsize}-#2*\value{boxsize}}
  \multiput(\value{tempdist},-\value{tempcounter})(0,\value{boxsize}){#1}\smbox
  \put(\value{tempdist},-\value{tempcounter}){%
    \makebox(\value{boxsize},\value{boxsize})[c]{$\smallentryformat 4$}}
  \setcounter{tempcounter}{\value{tempcounter}-\value{boxsize}}
  \put(\value{tempdist},-\value{tempcounter}){%
    \makebox(\value{boxsize},\value{boxsize})[c]{$\smallentryformat 2$}}
  \setcounter{tempcounter}{\value{tempcounter}-\value{boxsize}}
  \put(\value{tempdist},-\value{tempcounter}){%
    \makebox(\value{boxsize},\value{boxsize})[c]{$\smallentryformat 1$}}
  \setcounter{tempdist}{\value{tempdist}+\value{boxsize}}
  \multiput(\value{tempdist},\value{tempe})(0,\value{boxsize}){#2}\smbox
  \put(\value{tempdist},\value{tempe}){%
    \makebox(\value{boxsize},\value{boxsize})[c]{$\smallentryformat 3$}}
  \setcounter{tempe}{\value{tempe}+\value{boxsize}}
  \put(\value{tempdist},\value{tempe}){%
    \makebox(\value{boxsize},\value{boxsize})[c]{$\smallentryformat 2$}}
  \setcounter{tempe}{\value{tempe}+\value{boxsize}}
  \put(\value{tempdist},\value{tempe}){%
    \makebox(\value{boxsize},\value{boxsize})[c]{$\smallentryformat 1$}}
}

\newtheoremstyle{mytheorems}{9pt}{6pt}{\itshape}{0pt}{\sc}{.}{ }{}
\newtheoremstyle{myremarks}{6pt}{3pt}{\normalfont}{0pt}{\it}{.}{ }{}

\theoremstyle{mytheorems}
\newtheorem{theorem}{Theorem}
\newtheorem{lemma}{Lemma}
\newtheorem{corollary}{Corollary}
\newtheorem{proposition}{Proposition}

\theoremstyle{myremarks}
\newtheorem*{notation}{Notation}
\newtheorem*{example}{Example}

\newtheorem*{remark}{Remark}
\newtheorem*{definition}{Definition}

\renewcommand\qed{\phantom{m.} $\!\!\!\!\!\!\!\!$\nolinebreak\hfill\checkmark}

\DeclareMathOperator\type{{\rm type}}
\DeclareMathOperator\ind{{\rm ind}}
\DeclareMathOperator\soc{{\rm soc}}

\DeclareMathOperator\mepi{{\rm mepi}}
\DeclareMathOperator\rad{{\rm rad}}

\DeclareMathOperator\Hom{{\rm Hom}}
\DeclareMathOperator\End{{\rm End}}
\DeclareMathOperator\Ext{{\rm Ext}}
\DeclareMathOperator\Aut{{\rm Aut}}
\DeclareMathOperator\Cok{{\rm Cok}}
\DeclareMathOperator\up{\!\uparrow}
\DeclareMathOperator\down{\!\downarrow}

\renewcommand\Im{{\rm Im}}
\DeclareMathOperator\len{{\rm len}}
\newcommand\onto{\twoheadrightarrow}

\begin{document}
{\footnotesize [s-proto-4g.tex, \today]}
\vglue1truecm
\centerline{\Large Hall Polynomials}
\medskip\centerline{\Large via Automorphisms of Short Exact Sequences}
\bigskip
\centerline{By}
\medskip
\centerline{\sc Markus Schmidmeier}
\bigskip
\centerline{\rm Dedicated to Wolfgang Zimmermann}

\bigskip\medskip

\centerline{\parbox{10cm}{\footnotesize 
    {\it Abstract.}  
    We present a sum-product formula for the classical Hall polynomial 
    which is based on tableaux that have been introduced by T.\ Klein in 1969.
    In the formula, each summand corresponds to a Klein tableau, while the product is
    taken over the cardinalities of automorphism groups of short exact sequences which
    are derived from the tableau.
    For each such sequence, one can read off from the tableau the summands in an 
    indecomposable decomposition, and the size of their homomorphism and automorphism
    groups.
    Klein tableaux are refinements of Littlewood-Richardson tableaux in the sense that 
    each entry $\ell\geq2$ carries a subscript $r$. 
    We describe module theoretic and categorical properties shared by short exact sequences
    which have the same symbol $\ell_r$ in a given row in their Klein tableau.
    Moreover, we determine the interval in the Auslander-Reiten quiver in which
    indecomposable sequences of $p^n$-bounded groups 
    which carry such a symbol occur.
}}
\renewcommand{\thefootnote}{}
\footnotetext{{\it MSC 2010:}   5E10, 16G70, 20K27}
\footnotetext{{\it Affiliation:} Math.\ Sciences, Florida Atlantic Univ., 
                              Boca Raton, Florida 33431}
\footnotetext{{\it Contact:}  E-mail: {\tt mschmidm@fau.edu}, 
                          Tel: 1-561-297-0275, FAX: 1-561-297-2436}

\bigskip

The short exact sequences $E: 0\to A\to B\to C\to 0$ of finite abelian $p$-groups 
form the objects in the category $\mathcal S$; morphisms are
the commutative diagrams.
{\it Prototypes} or {\it  Klein tableaux,} as we call them in this paper,
were introduced in \cite{klein2}
as an isomorphism invariant for the objects in $\mathcal S$.
This invariant is finer than the partition triple $(\alpha,\beta,\gamma)$
which consists of the types of the groups $A,B,C$; 
and even finer than the Littlewood-Richardson (LR-) tableau
associated with $E$.
More precisely, the Klein tableau has the same entries as the LR-tableau, but
each entry bigger than one carries a subscript.  We recall that the largest
entry is the exponent $e$ of $A$.

\smallskip
It turns out that the Klein tableaux with entries at most two are in one-to-one
correspondence with the isomorphism types of short exact sequences with first term
$p^2$-bounded (Proposition~\ref{proposition-KRS-multiplicity}).
For an arbitrary sequence $E$ we denote the corresponding Klein tableau by $\Pi(E)$,
and for a Klein tableau $\Pi$ with entries at most 2 the corresponding sequence 
by $E(\Pi,p)$. 
One can not expect that Klein tableaux classify arbitrary short exact sequences
up to isomorphism since even for $e=3$ 
there occur parametrized families of pairwise non-isomorphic sequences
which have the same Klein tableau \cite{bounded}.

\smallskip
In fact, the combinatorial data contained in the Klein tableau $\Pi$
corresponding to $E$ describe exactly the
isomorphism types of the sequences
$$E|_2^\ell:\qquad 0\;\to\; p^{\ell-2}A/p^\ell A\;\to\; B/p^\ell A\;\to\; B/p^{\ell-2}A\;\to\; 0$$
(and hence also of the sequences
$$E|_1^\ell:\qquad 0\;\to\; p^{\ell-1}A/p^\ell A\;\to\; B/p^\ell A\;\to\; B/p^{\ell-1}A\;\to\; 0).$$

We will see that the tableaux $\Pi|_i^\ell=\Pi(E|_i^\ell)$ corresponding
to these sequences are obtained from $\Pi$ as suitable restrictions ($\ell\geq0, i=1,2$).

\smallskip
Given finite abelian $p$-groups $A,B,C$ of type $\alpha,\beta,\gamma$,
respectively, the classical Hall polynomial $g_{\alpha,\gamma}^\beta(p)$
counts the number of subgroups $U$ of $B$ such that  $U\cong A$
and $B/U\cong C$. 
Often, Hall polynomials are computed using LR-tableaux, see for example
\cite{macdonald},
but in earlier articles, the computation is based on Klein tableaux \cite{klein2}.

\smallskip
In the sum-product formula for Hall polynomials presented in this paper,
Klein tableaux control the counting process:  The sum is indexed by all
Klein tableaux of the given type, and in each summand all the short exact sequences
are determined uniquely, up to isomorphism, by suitable restrictions of the
corresponding tableau.

\begin{theorem}\label{theorem-Hall}
For partitions $\alpha,\beta,\gamma$, the classical Hall polynomial
can be computed as
$$g_{\alpha,\gamma}^\beta(p)\quad=\quad
   \sum_{\Pi} \; \prod_{\ell=2}^{e+1} \; 
         \frac{\#\Aut_{\mathcal S}E(\Pi|_1^\ell;p)}{%
               \#\Aut_{\mathcal S}E(\Pi|_2^\ell;p)}  $$
where the sum is taken over all Klein tableaux $\Pi$ of type $(\alpha,\beta,\gamma)$
and $e=\alpha_1$ is the exponent of the subgroup.
\end{theorem}

Each Klein tableau can be realized by short exact sequences; if the sequence $E$
has tableau $\Pi$ then the summand corresponding to $\Pi$ in the above formula
can be written as
$$g(\Pi;p)\;=\;
   \prod_{\ell=2}^{e+1}\frac{\#\Aut_{\mathcal S} E|^\ell_1}{\#\Aut_{\mathcal S} E|^\ell_2}.$$
This number counts the subgroups $U$ of $B$ such that the sequence 
$0\to U\to B\to B/U\to 0$ has Klein tableau $\Pi$.

\medskip
We describe module theoretic and categorical properties of the sequences corresponding
to a given Klein tableau.
We will see how the tableaux determine the size of certain homomorphism and automorphism
groups, in particular the size of the groups $\Aut_{\mathcal S} E|^\ell_i$, $i=1,2$, 
which occur in the formula.

\smallskip
Denote by $\mathcal S_2$ the full subcategory of 
$\mathcal S$ consisting of short exact sequences $E$ with $p^2A=0$.
The indecomposable objects in $\mathcal S_2$
are either {\it pickets,} i.e.\ sequences with cyclic middle term
of the form
$$P_\ell^m: \qquad 0\to (p^{m-\ell})\to 
                     \mathbb Z/(p^m)\to \mathbb Z/(p^{m-\ell})\to 0$$
where $\ell\leq \min\{m,2\}$.
Otherwise, they are {\it bipickets;} here the inclusion is a diagonal
embedding of $\mathbb Z/(p^2)$ in a direct sum of two cyclic $p$-groups.
\begin{eqnarray*} && \hspace{-5mm} T_2^{m,r}: \qquad \\
 && \hspace{-5mm} 0\to ((p^{m-2},p^{r-1}))\to
             \mathbb Z/(p^m)\oplus\mathbb Z/(p^r) \to
             \mathbb Z/(p^{m-1})\oplus\mathbb Z/(p^{r-1})\to 0\end{eqnarray*}
where $1\leq r\leq m-2$.
To unify notation, we put $T_2^{m,m-1}=P_2^m$.

\smallskip
Each object $T_2^{m,r}$ (where $1\leq r\leq m-1$, $(m,r)\neq (2,1)$)
occurs as end term of an Auslander-Reiten sequence in $\mathcal S_2$,
$$\mathcal A^{m,r}:\qquad 0\to X^{m,r}\to Y^{m,r}\stackrel{v_2^{m,r}}\to
                  T_2^{m,r}\to 0;$$
the remaining object $T_2^{2,1}=P_2^2$ is a projective object in $\mathcal S_2$.

\smallskip
Consider the lifting functor $\uparrow^i$ which maps a short exact sequence
$E:0\to A\to^fB\to C\to 0$ in $\mathcal  S$ to 
$$E\up^i: \qquad 0\;\to\; p^{-i}f(A)\;\subset\; B\;\to\; B/p^{-i}f(A)\;\to\; 0$$
where $p^{-i}f(A)=\{b\in B:p^ib\in f(A)\}$.

\smallskip
For $0\leq i\leq r-1\leq m-2$, the liftings $\mathcal A^{m,r}\up^i$ are 
short exact sequences, unless when $m=i+2, r=i+1$ in which case 
$T_2^{m,r}\up^i = P_m^m$ is a projective object in $\mathcal S_m$.

\smallskip
Suppose  that the Klein tableau $\Pi$ represents a short exact sequence $E$.
In our second theorem we interpret the entries in $\Pi$ in terms of the module
structure of $E$, and in terms of homological properties of $E$ as an object in the
category $\mathcal S$. 
Thus, the combinatorial data defining a Klein tableau have a precise algebraic
interpretation within the category $\mathcal S$ of short exact sequences.
In this sense, $\mathcal S$ provides a categorification for Klein tableaux.

\begin{theorem}\label{theorem-prototypes}
For a short exact sequence $E\in \mathcal S$ with Klein tableau $\Pi$
and natural numbers $\ell, m, r$ with $2\leq \ell\leq r+1\leq m$
the following numbers are equal.
\begin{enumerate}
\item The number of boxes $\singlebox{\ell_r}$ in the $m$-th row of $\Pi$.
\item The multiplicity of $T_2^{m,r}$ as a direct summand of $E|_2^\ell$.
\item The $\mathbb Z/(p)$-dimension of
$$\frac{\Hom_{\mathcal S}(E,T_2^{m,r}\up^{\ell-2})}{%
        \Im\Hom_{\mathcal S}(E,v_2^{m,r}\up^{\ell-2})}.$$
\end{enumerate}
\end{theorem}

The corresponding result for LR-tableaux is \cite[Theorem~1]{lr} 
where the entries in the tableau
are characterized in terms of the picket decomposition of the sequences $E|_1^\ell$, 
and in terms of spaces of homomorphisms from $E$ into pickets.

\smallskip
As a consequence of Theorem~\ref{theorem-prototypes}, we determine in 
Corollary~\ref{corollary-hom2bipicket} the size of the homomorphism groups
of the form $\Hom_{\mathcal S}(E,T_2^{m,r})$.

\smallskip
In Theorem~\ref{theorem-factor} we describe how all the sequences $E$ 
with Klein tableau containing a symbol $\singlebox{\ell_r}$ in the $m$-th
row can be detected within the category $\mathcal S(n)$ 
of short exact sequences of $p^n$-bounded finite abelian groups. 
Let $Z=T_2^{m,r}\up^{\ell-2}$ be as in the theorem.  We specify an object
$C$ depending only on $Z$ and $n$ such that for each sequence $E$,
the Klein tableau has a symbol $\singlebox{\ell_r}$ in the $m$-th row 
if and only if there are maps $f:C\to E$, $g:E\to Z$
with the property that the composition $gf$ does not factor through the sink map
for $Z$ in $\mathcal S_2\up^{\ell-2}$.
In this sense, the sequences $E$ which have a symbol $\singlebox{\ell_r}$
in the $m$-th row of their LR-tableau ``lie between'' $C$ and $Z$.

\bigskip
We describe the contents of the sections in this paper.

\smallskip
As a gentle introduction to Klein tableaux, we will review in Section~\ref{section-prototype} 
combinatorial
isomorphism invariants for short exact sequences.  We point out that LR-tableaux and
Klein tableaux are local invariants in the sense that they depend only on subfactors 
of the given sequence where the first term is $p^2$-bounded.

\medskip
In Section~\ref{section-categoryS2} we study the category $\mathcal S_2$.
To simplify notation, we consider the objects as 
embeddings $(A\subset B)$ of finite abelian $p$-groups 
where $A$ is $p^2$-bounded.
We show that Klein tableaux determine the decomposition of arbitrary
embeddings in $\mathcal S_2$ as  direct sums of pickets and bipickets 
(Proposition~\ref{proposition-KRS-multiplicity}),
and discuss the Auslander-Reiten quiver for this category.

\medskip
In Section~\ref{section-Hall} we give the proof of Theorem~\ref{theorem-Hall}.
The two main ingredients are a sum-product formula for Hall polynomials 
in \cite{klein2}, and our study of the action of the group $\Aut_{\mathbb Z}(B)$
on sequences of the form $0\to A\to B\to C\to 0$. 
We illustrate the computations in Theorem~\ref{theorem-Hall} in an example;
for this we use Corollary~\ref{corollary-hom2bipicket} to determine the numbers
$\#\Aut E|^\ell_i$, $i=1,2$.

\medskip
In Section~\ref{section-categorification} 
we discuss how Klein tableaux determine the position of 
short exact sequences within the category $\mathcal S$.
We give the proofs for Theorems~\ref{theorem-prototypes} and \ref{theorem-factor},
and illustrate both results with examples in the category $\mathcal S(5)$.

\smallskip
For results and terminology regarding Auslander-Reiten sequences 
and approximations, we refer the reader to \cite{ars} and \cite{as}.

\medskip
{\it Acknowledgements.} The results in this paper have been obtained while the author
was spending parts of his sabbatical leave in 2008/09 at 
the Collaborative Research Center 701 at Universit\"at Bielefeld,
and then at the Mathematische Institut der Universit\"at zu K\"oln.
He would like to thank C.~M.~Ringel and S.~K\"onig for invitation and advice.

\smallskip
{\it Dedication.} 
The author would like to use the occasion given by the graduation of his first Ph.D.\ student
to thank W.~Zimmermann, the advisor of his 1996 dissertation, for all those skills and
attitudes which have served as foundation for the author's research and for 
his work with students.

\section{Klein tableaux} \label{section-prototype}

In this section we review the following combinatorial isomorphism invariants
for short exact sequences:
\begin{itemize}
\item Partition triples,
\item Littlewood-Richardson tableaux, and
\item Klein tableaux.
\end{itemize}

\begin{notation}
Let $R$ be a commutative principal ideal domain, $p$ a generator 
of a maximal ideal and $k=R/(p)$ the residue field.
A $p$-{\it module} is a finite length $R$-module which is 
annihilated by some power of $p$.

\smallskip
We denote by $\mathcal S$ the category of all short exact sequences 
$$E:\quad 0\to A\to B\to C\to 0$$
of $p$-modules, with morphisms given by commutative diagrams. This category
is equivalent to the category of embeddings $E:(A\subset B)$
of $p$-modules, with morphisms given by commutative squares.
The symbol $\mathcal S$ denotes either one of those
categories.
For natural numbers $\ell,n$, let $\mathcal S_\ell$ and $\mathcal S(n)$ be
the full subcategories of $\mathcal S$ of all embeddings $(A\subset B)$ 
which satisfy the conditions $p^\ell A=0$ and $p^n B=0$, respectively.

\subsection{The partition triple}\label{section-partition-triple}

We denote the indecomposable $p$-module of composition length $m$ by $P^m=R/(p^m)$.
It is well known that arbitrary $p$-modules are given by partitions:
\end{notation}

\begin{proposition}\label{proposition-partitions}
There is a one-to-one correspondence
$$   \big\{\text{$p$-modules}\big\}{\big/}_{\textstyle\cong}
     \quad\stackrel{1-1}{\longleftrightarrow}\quad
     \big\{\text{partitions}\big\} . $$
The partition $\beta=(\beta_1,\ldots,\beta_s)$ corresponds to the $p$-module 
$M(\beta)=\bigoplus_{i=1}^sP^{\beta_i}$.
Conversely, given a $p$-module $B$, its {\it type}
$\beta=\type(B)$ is obtained via the formula
$$\beta'_i=\dim_k{\displaystyle\frac{p^{i-1}B}{p^iB}}\quad\text{for $i\in\mathbb N$}$$
where $\beta'$ is the conjugate of $\beta$.

The multiplicity of $P^m$ in an indecomposable decomposition of $B$ is
$$\mu_{P^m}(B)  \;=\;  \#\{ \; i \;|\; \beta_i=m\; \}  \; =\;  \beta'_m-\beta'_{m+1}.$$
  \qed
\end{proposition}

We picture $P^m$ as a column of $m$ boxes since the parts of a partition 
will be given by the lengths of the columns in its diagram.

\begin{definition}
Given a short exact sequence $E: 0\to A\to B\to C\to 0$ of $p$-modules, 
the {\it partition triple} consists of the three partitions
$$\big(  \type(A),\; \type(B),\; \type(C) \big).$$
\end{definition}

Clearly, the partition triple forms an isomorphism invariant for the objects 
in $\mathcal S$.

\begin{example}
In the embedding
$$T_2^{4,2}\;=\;
      \big(A\subset B\big) \;=\; \big( ((p^2,p))\subset P^4\oplus P^2\big),$$
the submodule $A$ is cyclic of exponent $2$, so $\alpha=\type(A)=(2)$ and
$\beta=\type(B)=(4,2)$.  
Note that the factor $B/A$ is not annihilated by $p^2$, 
hence it has type $\gamma=\type(B/A)=(3,1)$;
thus, the bipicket $T_2^{4,2}$ has partition triple $((2),(4,2),(3,1))$ or $(2,42,31)$ for
short.

\smallskip
In general, 
the partition triple for the picket $P_\ell^m$ is $(\ell,m,m-\ell)$, while the 
partition triple for the bipicket $T_2^{m,r}$ is $((2),(m,r),(m-1,r-1))$. 
\end{example}

\subsection{Littlewood-Richardson tableaux}

According to theorems by Green
and Klein \cite[Section~4]{klein}, a triple of partitions $(\alpha,\beta,\gamma)$ 
can be realized as the partition triple
of some embedding $E\in\mathcal S$ if and only if there is an LR-tableau $\Gamma$ 
of type $(\alpha,\beta,\gamma)$.

\begin{definition}
A weakly increasing sequence of partitions $\Gamma=[\gamma^0,\ldots,\gamma^e]$ 
forms a {\it Littlewood-Richardson tableau (LR-tableau)}
provided the following conditions hold:
\begin{enumerate}
\item For each $1\leq \ell\leq e$, the skew tableau 
$\gamma^\ell\backslash\gamma^{\ell-1}$
forms a {\it horizontal stripe,} that is, 
$\gamma^\ell_i-\gamma^{\ell-1}_i\leq 1$ holds for each $i$.
\item The {\it lattice permutation property} is satisfied, that is, we have for each
 $2\leq \ell\leq e$ and each $k\geq 0$:
$$\sum_{i\geq k} \big(\gamma^\ell_i-\gamma^{\ell-1}_i\big)\;
    \leq \; \sum_{i\geq k} \big(\gamma^{\ell-1}_i-\gamma^{\ell-2}_i\big).$$
\end{enumerate}
Let $\alpha$ be the conjugate of the partition defined by the lengths 
of the horizontal stripes, that is, $\alpha'_\ell=\sum_i(\gamma^\ell_i-\gamma^{\ell-1}_i)$.
Let $\beta=\gamma^e$ and $\gamma=\gamma^0$.  Then we say that the LR-tableau has
{\it type} $(\alpha,\beta,\gamma)$.
\end{definition}

\smallskip
The following observation is immediate:

\begin{lemma}\label{lemma-observation-LR}
Let $e\geq 2$.  
A weakly increasing sequence $\Gamma=[\gamma^0,\ldots,\gamma^e]$ of partitions
has the LR-property if and only if 
each  restriction $\Gamma|_2^\ell=[\gamma^{\ell-2},\gamma^{\ell-1},\gamma^\ell]$
where $2\leq \ell \leq e$ does.
                                   \qed
\end{lemma}

We picture the tableau $\Gamma$ as the diagram $\beta=\gamma^e$ 
in which for each $\ell\geq 1$ the horizontal stripe
$\gamma^\ell\backslash\gamma^{\ell-1}$ is filled with boxes $\singlebox \ell$.

\begin{example}
The LR-sequence $\Gamma=[21,321,332,432]$ and its restriction 
$\Gamma|_2^3=[321,332,432]$ have the following LR-tableaux.
$$
\setcounter{boxsize}{3}
\beginpicture\setcoordinatesystem units <1.1cm,1.1cm>
\put {} at 0 1.5
\put {} at 7 .5
\put{$\Gamma:$} at .5 1
\put{\begin{picture}(9,12)\put(0,7){\boxes31}\put(0,0){\numbox 3}
                          \put(3,7){\boxes33}
                          \put(6,9){\boxes23}
     \end{picture}} at 2 1
\put{$\Gamma|_2^3:$} at 4.5 1
\put{\begin{picture}(9,12)\put(0,6){\boxes42}
                          \put(3,7){\boxes31}
                          \put(6,9){\boxes21}
     \end{picture}} at 6 1
\endpicture
$$
\end{example}

\subsection{The LR-tableau of an embedding}\label{subsection-LR}

For an embedding $E:(A\subset B)$ of $p$-modules, the corresponding LR-tableau 
is obtained as follows.
Let  $(\alpha,\beta,\gamma)$ be the partition triple for $E$ and let
$e=\alpha_1$ be the exponent of $A$.
The chain of inclusions
$$0=p^eA\;\subset\;p^{e-1}A\;\subset\;\cdots\;\subset\;p^0A\;=\;A$$
yields a chain of epimorphisms
$$B=B/p^eA\;\onto\; B/p^{e-1}A\;\onto\;\cdots\;\onto\;B/p^0A\;=\;B/A$$
and hence a weakly decreasing sequence of partitions
$$\beta=\gamma^e\;\geq\;\gamma^{e-1}\;\geq\;\cdots\;\geq\;\gamma^0=\gamma$$
where $\gamma^i=\type(B/p^iA)$.  Then $\Gamma=[\gamma^0,\ldots,\gamma^e]$ 
is the LR-tableau for $E$ \cite[Theorem~4.1]{klein}.

\smallskip
The LR-tableau $\Gamma$ is an isomorphism invariant for $E$ refining the 
partition triple; in fact, the type of $\Gamma$ is the partition triple 
$(\alpha,\beta,\gamma)$ for $E$.  In \cite{lr} we give an interpretation for each entry
in $\Gamma$ in terms of the direct sum decomposition of the subfactors
$$E|_1^\ell\;=\;\big(p^{\ell-1}A/p^\ell A \;\subset\; B/p^\ell A\big)\;\in\mathcal S_1,$$
where $1\leq \ell\leq e$, 
and in terms of homomorphisms in the category $\mathcal S$.

\begin{example}
The LR-tableau for the picket $P_\ell^m$ is easily computed as $[m-\ell,m-\ell+1,\ldots,m]$.

\smallskip
We compute the LR-tableau for the bipicket $T_2^{42}:(A\subset B)$ from
example (\ref{section-partition-triple});
for this we need the types of the factors $B/p^\ell A$. 
From the partition triple we read off that $\type(B/A)=31$.
Since 
$$(pA\subset B)=\big(((p^3,0))\subset P^4\oplus P^2\big)$$
we obtain that $\type(B/pA)=32$.  Clearly, $\type(B/p^2A)=\type(B)=42$.
So the LR-tableau 
is $\Gamma=[31, 32, 42]$, as pictured below.

\smallskip
We note that both $P_2^4\oplus P_0^3\oplus P_1^2$ and $T_2^{42}\oplus P_1^3$ have
the same LR-tableau $\Gamma'$ , while their Klein tableaux are different as we will see
in (\ref{section-entries-at-most-2}).
$$
\setcounter{boxsize}{3}
\beginpicture\setcoordinatesystem units <1.1cm,1.1cm>
\put {} at 0 0
\put {} at .5 .5
\put{$\Gamma:$} at 0 0
\put{\begin{picture}(9,12)\put(0,6){\boxes42}
                          \put(3,9){\boxes21}
     \end{picture}} at 1 0
\put{$\Gamma':$} at 3 0
\put{\begin{picture}(9,12)\put(0,6){\boxes42}
                          \put(3,7){\boxes31}
                          \put(6,9){\boxes21}
     \end{picture}} at 4 0
\endpicture
$$
\end{example}

\subsection{Klein tableaux}\label{section-Klein-tableaux}

In \cite[Section~1]{klein2} Klein introduces prototypes
(which we call Klein tableaux) as refinemenents
of LR-tableaux.  Here we use subscript functions for an efficient encoding of
the data in the tableau.

\begin{definition}
Let $\Gamma=[\gamma^0,\ldots,\gamma^e]$ be a weakly increasing sequence of 
partitions and let $2\leq \ell\leq e$. 
An $\ell$-{\it subscript function} is a map
$$\varphi^\ell:\gamma^\ell\backslash\gamma^{\ell-1}\;\longrightarrow\;\mathbb N,$$
defined on the set of boxes in the skew tableau $\gamma^\ell\backslash\gamma^{\ell-1}$ 
such that the following conditions are satisfied:
\begin{itemize}
\item[(i)] In each given row, the map $\varphi^\ell$ is weakly increasing.
\item[(ii)] If a box $b$ occurs in the $m$-th row, then $\varphi^\ell(b)< m$.
\item[(iii)] If a box $b$ lies in the $m$-th row,
  and the  box above $b$ is in $\gamma^{\ell-1}\backslash\gamma^{\ell-2}$, 
  then $\varphi^\ell(b)=m-1$.
\item[(iv)] There are at least $\#(\varphi^\ell)^{-1}(r)$ boxes in the $r$-th row of 
  $\gamma^{\ell-1}\backslash\gamma^{\ell-2}$.
\end{itemize}
\end{definition}

\smallskip
The data
$$\Pi=[\gamma^0,\ldots,\gamma^e; \varphi^2,\ldots,\varphi^e]$$
define a {\it Klein sequence} if $\Gamma=[\gamma^0,\ldots,\gamma^e]$ is an
LR-sequence, and if for each $2\leq\ell\leq e$ the map $\varphi^\ell$ is an 
$\ell$-subscript function.
We say that the Klein sequence {\it refines} the LR-sequence; we define 
its {\it type} to be the type of the LR-sequence.

\smallskip
The {\it Klein tableau} represents the data in the Klein sequence as follows.
In the LR-tableau $\Gamma$ replace each box $b$ with an entry $\ell\geq 2$ 
by the {\it symbol\/} $\singlebox{\ell_r}$ where $r=\varphi^\ell(b)$.
We call $\ell$ the {\it entry} or {\it label\/} of the box $b$, $r$ the {\it subscript\/},
and $\singlebox{\ell_r}$ the {\it symbol.} 
Usually we will omit  subscripts that are uniquely  determined.

\begin{remark}
We recall the following equivalent definition for a Klein tab\-leau 
from \cite{klein2}.  An LR-tableau where each entry $\ell\geq 2$
carries a subscript is a Klein tableau provided the following conditions are satisfied:
\begin{itemize}
  \item[(i)] In any row, the subscripts of the same entry weakly increase from left to right.
  \item[(ii)] The subscript of an entry $\ell\geq 2$ in row $m$ is at most $m-1$.
  \item[(iii)] Any entry $\ell$ occuring in the same column as an entry $\ell-1$
    must carry the subscript $m-1$ where $m$ is  the row of the entry $\ell$.
  \item[(iv)] The total number of symbols $\singlebox{\ell_r}$ cannot exceed the number
    of $\ell-1$'s in row $r$. 
\end{itemize}
\end{remark}

\smallskip
{\it Motivation.}
Let $2\leq \ell\leq s$.  
The lattice permutation property in the definition of an LR-tableau
makes sure that there is an injective map 
$$\psi^\ell:\; \big\{\text{boxes labelled $\ell$}\big\}
  \;\longrightarrow\; \big\{\text{boxes labelled $\ell-1$}\big\}$$
such that each $\psi^\ell(b)$ occurs in some row above $b$.
The Klein tableau encodes a normalized version of $\psi^\ell$, as follows.
The map $\varphi^\ell$ given by $\varphi^\ell(b)={\rm row}(\psi^\ell(b))$ 
will satisfy (ii) and (iv).
Given that $\varphi^\ell$ satisfies (ii) and (iv), then there exists
a possibly modified version of $\varphi^\ell$ which will satisfy in addition (iii). 
In order to make the map weakly increasing (i), compose it with a permutation
of boxes in the same row and with the same entry.
Note that (iii) will still be satisfied.

\begin{definition}
For $\Pi=[\gamma^0,\ldots,\gamma^e;\varphi^2,\ldots,\varphi^e]$ a Klein tableau
and for natural numbers $u\leq\ell\leq e$ define the {\it restrictions} 
$$\Pi|^\ell=[\gamma^0,\ldots,\gamma^\ell;\varphi^2,\ldots,\varphi^\ell]\;\text{and}\;
\Pi|^\ell_u=[\gamma^{\ell-u},\ldots,\gamma^\ell;\varphi^{\ell-u+2},\ldots,\varphi^\ell].$$
\end{definition}

The Klein tableau for the restriction $\Pi|^\ell_u$ is the skew tableau
of shape $\gamma^\ell\backslash\gamma^{\ell-u}$;
each of its entries is obtained from the corresponding entry in $\Pi$
by subtracting $\ell-u$.  The new entries 1 loose their subscript, 
while each remaining entry $x\geq 2$ in $\Pi|^\ell_u$ inherits its
subscript from the corresponding entry $x+\ell-u$ in $\Pi$.

\smallskip
We observe as in Lemma~\ref{lemma-observation-LR}:

\begin{lemma}
Suppose $[\gamma^0,\ldots,\gamma^e]$ is a weakly increasing sequence of partitions
where $e\geq 2$,
and there are maps $\varphi^\ell:\gamma^\ell\backslash\gamma^{\ell-1}\to \mathbb N$
for each $2\leq\ell\leq e$. Then the 
system $\Pi=[\gamma^0,\ldots,\gamma^e;\,\varphi^2,\ldots,\varphi^e]$ 
is a Klein tableau if and only if for each $2\leq \ell\leq e$ the 
restriction
$$\Pi|_2^\ell\;=\;[\gamma^{\ell-2},\gamma^{\ell-1},\gamma^\ell;\,\varphi^\ell]$$
is a Klein tableau.
                                  \qed
\end{lemma}

\begin{example}
For the LR-tableau $\Gamma$ below, there is only one subscript function $\varphi^2$ because
of condition (iii).  However, there are two subscript functions $\varphi^3$ as the
entry 3 can have either subscript 2 or 3.
Hence there are the two Klein-tableaux, $\Pi$ and $\Pi'$, which refine $\Gamma$.

$$
\setcounter{boxsize}{3}
\beginpicture\setcoordinatesystem units <1cm,1cm>
\put {} at 0 .5
\put {} at 7 .5
\put{$\Gamma:$} at 0 0
\put{\begin{picture}(9,12)\put(0,7){\boxes31}\put(0,0){\numbox 3}
                          \put(3,7){\boxes33}
                          \put(6,9){\boxes23}
     \end{picture}} at 1 0
\put{$\Pi:$} at 2.5 0
\put{\begin{picture}(9,12)\put(0,7){\boxes31}\put(0,0){\numbox{3_2}}
                          \put(3,9){\boxes21}\put(3,3){\numbox{2_2}}
                          \put(6,10){\boxes11}\put(6,6){\numbox{2_1}}
     \end{picture}} at 3.5 0
\put{$\Pi':$} at 5 0
\put{\begin{picture}(9,12)\put(0,7){\boxes31}\put(0,0){\numbox {3_3}}
                          \put(3,9){\boxes21}\put(3,3){\numbox{2_2}}
                          \put(6,10){\boxes11}\put(6,6){\numbox{2_1}}
     \end{picture}} at 6 0
\endpicture
$$
\end{example}

\subsection{The Klein tableau of an embedding}\label{subsection-Klein}

Let $E:(A\subset B)$ be an embedding, say $E\in\mathcal S(n)$,
with LR-tableau $\Gamma=[\gamma^0,\ldots,\gamma^e]$.
The following partition sequence will  define the Klein tableau $\Pi=\Pi(E)$ 
corresponding to $E$ \cite[Theorem~2.3]{klein2}.

\smallskip
Given $\ell \geq 2$, the chain of submodules
\begin{eqnarray*}
  p^\ell A=p^\ell A+p(p^{\ell-2}A\cap p^{n-1}B) & \subset &
           p^\ell A+p(p^{\ell-2}A\cap p^{n-2}B) \\
          & \subset & \cdots \\
          & \subset & p^\ell A+p(p^{\ell-2}A\cap B)=p^{\ell-1}A
\end{eqnarray*}
yields a chain of epimorphisms
\begin{eqnarray*}
 \frac B{p^\ell A} = \frac B{p^\ell A+p(p^{\ell-2}A\cap p^{n-1}B)} & \onto &
                                \frac B{p^\ell A+p(p^{\ell-2}A\cap p^{n-2}B)} \\
                                & \onto & \cdots \\
                                & \onto & 
                                \frac B{p^\ell A+p(p^{\ell-2}A\cap B)} =
                                \frac B{p^{\ell-1}A}
\end{eqnarray*}
and hence a weakly decreasing chain of partitions
$$\gamma^\ell=\gamma^{\ell, n-1}\geq \gamma^{\ell,n-2} \geq \cdots \geq 
               \gamma^{\ell,0}=\gamma^{\ell-1}$$
where 
$$\gamma^{\ell,r} = 
       \type\left(\frac B{p^\ell A + p(p^{\ell-2}A\cap p^r B)}\right).$$

\smallskip
The Klein tableau $\Pi$ corresponding to $E$ will have symbols 
$\singlebox{\ell_r}$
in the skew tableau $\gamma^{\ell,r}\backslash\gamma^{\ell,r-1}$, 
for $2\leq \ell$ and $1\leq r<n$, 
and symbols $\singlebox 1$ in the skew tableau $\gamma^1\backslash\gamma^0$.

\smallskip
Equivalently, the Klein sequence corresponding to $E$ is given by the LR-sequence 
$\Gamma$ and the $\ell$-subscript functions $\varphi^\ell$ which are defined via
$$\varphi^\ell(b) = 
        r \quad \text{if}\; 
                    b\in\gamma^{\ell,r}\backslash\gamma^{\ell,r-1}\;
	\qquad ( 1\leq r<n, 2\leq \ell).$$

\begin{example}
The Klein tableau for a picket $P_\ell^m$ is determined by condition (iii):
The subscript of an entry $\ell\geq 2$ in row $m$ is $m-1$.

\smallskip
We compute the Klein tableau for the bipicket $T_2^{4,2}$ considered in the examples
in (\ref{section-partition-triple}) and in (\ref{subsection-LR}) where
we have seen that the LR-tableau $\Gamma$ for $T_2^{4,2}$ is as pictured below.

\smallskip
For the Klein tableau it remains to determine the subscript of the entry 2.
With $\ell=2$, we have $p^\ell A=0$ and $p^{\ell-2}A=A$, so the chain of
epimorphisms simplifies as
$$B=\frac B{p^2A}\onto \frac B{p(A\cap p^2B)}\stackrel{(*)}{\onto}
    \frac B{p(A\cap pB)}\onto \frac B{pA}.$$
As $A\subset pB$ but $A\not\subset p^2B$, the map labelled $(*)$ is the only proper
epimorphism and the partitions representing the above modules are $42,42,32,32$.
Thus, the subscript is $r=2$ and the symbol $\singlebox{2_2}$.
$$
\setcounter{boxsize}{3}
\beginpicture\setcoordinatesystem units <1.1cm,1.1cm>
\put {} at 1 1.5
\put {} at 6 .5
\put{$\Gamma:$} at 1 1
\put{\begin{picture}(3,0)\lboxes42\end{picture}} at 2 1
\put{$\Pi:$} at 5 1
\put{\begin{picture}(3,0)\pboxes42\end{picture}} at 6 1
\endpicture
$$
In this case, the subscript $r=2$ of the entry $\ell=2$ is uniquely determined 
by the tableau as the row of the corresponding entry $1$. 
Hence the subscript can be omitted.
\end{example}

\section{The Category $\mathcal S_2$}\label{section-categoryS2}

As a full exact subcategory of $\mathcal S$, 
the category $\mathcal S_2$ of all
embeddings $(A\subset B)$ where $A$ is $p^2$-bounded is 
itself an exact Krull-Remak-Schmidt category, so every object has a unique
direct sum decomposition into indecomposables. 
The indecomposable objects have been determined in \cite{bhw}, they are either
pickets or bipickets. 
We show that arbitrary objects
in $\mathcal S_2$ are determined uniquely, up to isomorphism, by their
Klein tableaux.  
We also compute the Auslander-Reiten quiver for $\mathcal S_2$.
It turns out that each indecomposable object is the starting term
of a source map; however, some indecomposable objects do not occur as
end terms of sink maps.

\subsection{Pickets and Bipickets}

Pickets and bipickets are  introduced for $R$-modules as for finite abelian groups.
It turns out that
each indecomposable object in $\mathcal S_2$ is either a picket or a bipicket
\cite[Theorem~7.5]{bhw}:

\begin{theorem}\label{ps}
The category $\mathcal S_2$ is an exact Krull-Remak-Schmidt category.
The indecomposable objects, up to isomorphism, are as follows.
$$\ind\mathcal S_2\;=\;
           \big\{ P_\ell^m \big| \ell\leq\min\{2,m\}\big\} \;\cup\;
           \big\{ T_2^{m,r} \big| 1\leq r\leq m-2\big\}.$$
\end{theorem}

\medskip
The Klein tableaux 
of those objects can be computed as in (\ref{subsection-Klein}):

\begin{center}
\setcounter{boxsize}{3}
  \begin{picture}(110,24)
    \put(0,12){$P_0^m:$}
    \put(10,9){\begin{picture}(3,12)\put(0,0)\smbox \put(0,6)\vdotbox
        \put(3,3){$\left.\makebox(0,5){}\right\}{\!}_m$}
      \end{picture}}
    \put(30,12){$P_1^m:$}
    \put(40,8){\begin{picture}(3,15)\put(0,3){\boxes21}\put(0,9)\vdotbox
        \put(3,5){$\left.\makebox(0,6){}\right\}{\!}_m$}
      \end{picture}}
    \put(60,12){$P_2^m:$}
    \put(70,6){\begin{picture}(3,18)\put(0,4){\boxes33}\put(0,12)\vdotbox
        \put(3,6){$\left.\makebox(0,8){}\right\}{\!}_m$}
      \end{picture}}
    \put(87,12){$T_2^{m,r}:$}
    \put(105,0){\begin{picture}(6,24)
        \put(-6,11){${}_{m\!\!}\left\{\makebox(0,13){}\right.$}
        \put(0,3){\boxes22}\put(0,9)\vdotbox
        \put(0,15){\boxes20}\put(3,15){\boxes21}\put(0,21)\vdotbox
        \put(3,21)\vdotbox\put(6,17){$\left.\makebox(0,6){}\right\}{\!}_r$}
      \end{picture}}
  \end{picture}
\end{center}

\subsection{Klein tableaux with entries at most 2}\label{section-entries-at-most-2}

In this section we show that there is a one-to-one correspondence 
between Klein tableaux with entries at most 2 
and isomorphism types of objects in $\mathcal S_2$.

\begin{proposition}\label{proposition-KRS-multiplicity}
There is a one-to-one correspondence
$$\begin{array}{ccc}
\big\{\text{objects in $\mathcal S_2$}\big\}{\big/}_{\textstyle\cong}  &
       \longleftrightarrow &
       \big\{\text{Klein tableaux with entries at most 2}\big\}, \\
   E  & \longmapsto & \Pi(E). \\
\end{array}
$$
Given an object $E\in\mathcal S_2$ with Klein tableau $\Pi$, 
the multiplicity $\mu_F(E)$ of an indecomposable object $F\in\mathcal S_2$
in a direct sum decomposition for $E$ is as follows:
$$\begin{array}{c | c}
  F & \mu_F(E) \\[.5ex] \hline
 T_2^{m,r\strut} &  \#\{\text{boxes $\singlebox{2_r}$ in row $m$}\}\qquad(r<m-1) \\[.5ex]
  P_2^m    &  
    \#\{\text{boxes $\singlebox{2_x}$ in row $m$}\} \qquad(x=m-1) \\[.5ex]
  P_1^m  & \#\{\text{boxes $\singlebox 1$ in row $m$}\} 
            -\#\{\text{boxes $\singlebox{2_m}$ in any row}\} \\[.5ex]
  P_0^m  & \#\{\text{empty col.\ of length $m$}\}
               + \#\{\text{boxes $\singlebox{}$ in row $m$ above $\singlebox{2_m}$}\}
\end{array}
$$

\end{proposition}

\smallskip
The corresponding result for arbitrary entries $\ell$ is
Corollary~\ref{corollary-multiplicity} where we give a module theoretic interpretation 
for the number of boxes $\singlebox{\ell_r}$ in row $m$.

\smallskip
For the proof of the proposition we use the following

\begin{lemma}\label{lemma-sum}
The Klein tableau corresponding to a direct sum
contains in each row all the symbols in lexicographical ordering
(with the empty symbol $\singlebox {}$ coming first) 
which occur in the corresponding rows in the tableaux of the summands.
\end{lemma}

\begin{proof}
Suppose that $E=\bigoplus E_i$ is a direct sum of embeddings 
in $\mathcal S_2$ and that the Klein tableaux for $E$ and for the $E_i$ 
are represented by partition sequences
$\Pi(E)=(\gamma^{\ell,r})$ and $\Pi(E_i)=(\gamma^{\ell,r}(i))$.
Given $\ell$, $r$, we can recover $\gamma^{\ell,r}$ as the type
of $F^{\ell,r}(E)$ where $F^{\ell,r}$ is the functor
$$F^{\ell,r}: \; \mathcal S \to R\text{-mod}, \;
            (A\subset B) \mapsto \frac B{p^\ell A+p(p^{\ell-2}A\cap p^r B)}. 
$$
Since $F^{\ell,r}$ is additive, the $m$-th row 
of the skew diagram $\gamma^{\ell,r}\backslash\gamma^{\ell,r-1}$ has length
$$(\gamma^{\ell,r})'_m - (\gamma^{\ell,r-1})'_m  \quad=\quad
       \sum(\gamma^{\ell,r}(i)'_m - \gamma^{\ell, r-1}(i)'_m).$$
Thus the number of symbols $\singlebox{\ell_r}$ in the $m$-th row in 
the Klein tableau $\Pi(E)$  is obtained by summing up the corresponding numbers
in the $\Pi(E_i)$. 
\end{proof}

\begin{proof}[Proof of the Proposition]
For each indecomposable object in $\mathcal S_2$, we have computed
the Klein tableau, and the Klein tableau for the sum is given by 
Lemma~\ref{lemma-sum}.  Hence the map $E\mapsto \Pi(E)$ is defined. 
It is onto since each Klein tableau
can be realized by an object in $\mathcal S$ \cite[Theorem~2.4]{klein2};
clearly, such an object must be in $\mathcal S_2$. 

\smallskip
It remains to demonstrate for given $E\in \mathcal S_2$ that the multiplicities
of the indecomposable summands of $E$ can be recovered from the Klein tableau
$\Pi=\Pi(E)$. 

\smallskip
Let $1\leq r\leq m-2$.  
The multiplicity of $T_2^{m,r}$ as a direct summand of $E$ 
equals the number of boxes $\singlebox{2_r}$ in row $m$.
Namely, $T_2^{m,r}$ is the only indecomposable object in $\mathcal S_2$ which has
this symbol in the given row in its Klein tableau. 

\smallskip
The multiplicity of the picket $P_2^{m}$ equals the number of boxes 
$\singlebox{2_x}$ in row $m$ where $x=m-1$.

\smallskip
The multiplicity of pickets of type $P_1^m$ 
is given by the number of ``unused'' 
boxes $\singlebox{1}$ in row $m$.  This number is the total number of 
such boxes, minus the number of symbols $\singlebox{2_m}$ in $\Pi$.

\smallskip
Finally, we deal with pickets of type $P_0^m$.  
Together they need to contribute
to $\Pi$ all the empty boxes which have not been obtained otherwise.
Given $m$, one summand $P_0^m$ has to be taken 
for each empty column of length $m$,
and also for each column of length $m+1$ 
which has only a symbol $\singlebox{2_m}$ in row $m+1$. 
(Such columns arise in direct sums like $P_2^{m+1}\oplus P_0^m$.)
\end{proof}

\begin{example}
We have seen in (\ref{subsection-LR}) that both $E:T_2^{42}\oplus P_1^3$
and $E':P_2^4\oplus P_0^3\oplus P_1^2$ have the same LR-tableau $\Gamma$.
The Klein tableaux of the indecomposable summands have been determined 
in (\ref{subsection-Klein}). Lemma~\ref{lemma-sum} yields the Klein tableaux
$\Pi=\Pi(E)$ and $\Pi'=\Pi(E')$ of the sums:
$$
\setcounter{boxsize}{3}
\beginpicture\setcoordinatesystem units <1cm,1cm>
\put {} at 0 .5
\put {} at 7 .5
\put{$\Gamma':$} at 0 0
\put{\begin{picture}(9,12)\put(0,6){\boxes42}
                          \put(3,7){\boxes31}
                          \put(6,9){\boxes21}
     \end{picture}} at 1 0
\put{$\Pi:$} at 2.5 0
\put{\begin{picture}(9,12)\put(0,7){\boxes30}\put(0,0){\numbox{2_2}}
                          \put(3,7){\boxes31}
                          \put(6,9){\boxes21}
     \end{picture}} at 3.5 0
\put{$\Pi':$} at 5 0
\put{\begin{picture}(9,12)\put(0,7){\boxes30}\put(0,0){\numbox {2_3}}
                          \put(3,7){\boxes31}
                          \put(6,9){\boxes21}
     \end{picture}} at 6 0
\endpicture
$$
Conversely, using Proposition~\ref{proposition-KRS-multiplicity} we can retrieve the direct
sum decompositions from the Klein tableaux.
\end{example}

\subsection{Auslander-Reiten sequences in $\mathcal S_2$}
\label{section-AR-sequences}

Let $\mathcal S_2(n) = \mathcal S_2\cap \mathcal S(n)$ be the full subcategory
of $\mathcal S$ of all pairs $(A\subset B)$ which satisfy $p^2A=0$ and $p^nB=0$.
For each $n$, the category $\mathcal S_2(n)$ has 
Auslander-Reiten sequences \cite{bounded}.  
It is the aim of this section to describe the Auslander-Reiten theory for $\mathcal S_2$.

\smallskip
Let $\mathcal E$ be an Auslander-Reiten sequence in $\mathcal S_2(n)$ with modules in 
$\mathcal S_2(n-1)$,
and let $n'\geq n$.  It follows from the description of the Auslander-Reiten quivers in 
\cite{bounded} that $\mathcal E$ is also an Auslander-Reiten sequence 
in $\mathcal S_2(n')$.  Hence, such a sequence $\mathcal E$ 
is an Auslander-Reiten sequence even in $\mathcal S_2$. 

\smallskip
More precisely, each indecomposable module in $\mathcal S_2$ has a source map
in $\mathcal S_2$, and each indecomposable object not of type $P_1^m$
has a sink map in $\mathcal S_2$.
The pickets of the form $P_1^m$ for $m\geq1$ are end terms 
of Auslander-Reiten sequences in each
of the categories $\mathcal S_2(n)$ for $n\geq m$, but those sequences 
depend on $n$.
We label those objects with $\times$ since they are neither projective, nor do they
occur as end terms of sink maps in $\mathcal S_2$. 
Finally, the module $P_2^2$ is an ($\Ext$-) projective object 
in  each of the categories $\mathcal S_2(n)$,
hence also in $\mathcal S_2$.

\smallskip 
Here is the partial Auslander-Reiten quiver for $\mathcal S_2$;
we picture each object by its Klein tableau:


$$
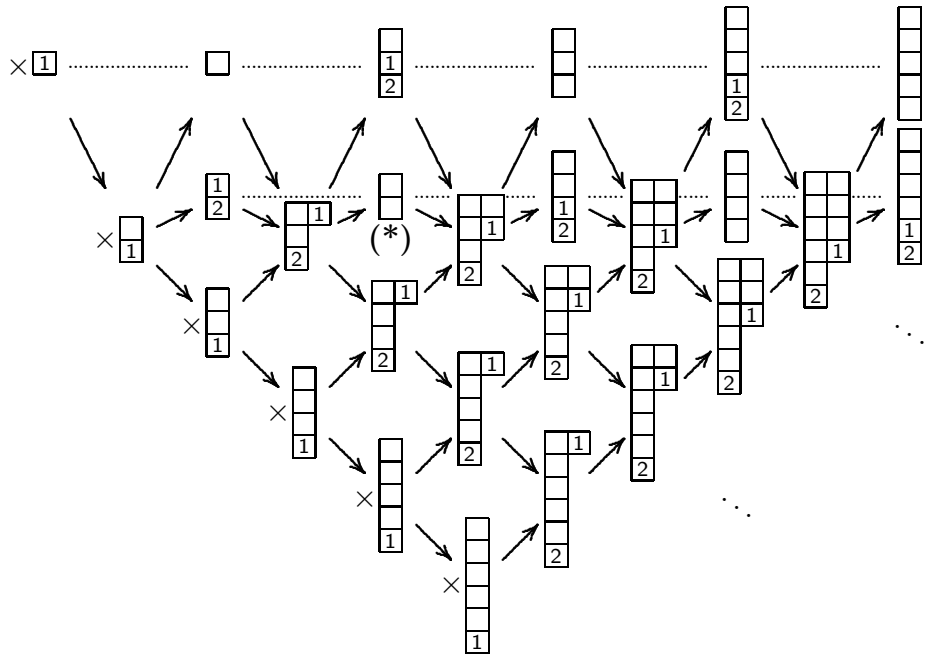
\beginpicture\setcoordinatesystem units <1.15cm,1.15cm>
\setcounter{boxsize}{3}
\put {} at 0 6.5
\put {} at 0 -.8
\put{\begin{picture}(3,0)\boxes11\end{picture}} at 0 6
\put{$\times$} at -.3 6
\put{\begin{picture}(3,0)\boxes21\end{picture}} at 1 4
\put{$\times$} at .7 4
\put{\begin{picture}(3,0)\boxes10\end{picture}} at 2 6
\put{\begin{picture}(3,0)\boxes23\end{picture}} at 2 4.5
\put{\begin{picture}(3,0)\boxes31\end{picture}} at 2 3
\put{$\times$} at 1.7 3
\put{\begin{picture}(3,0)\lboxes31\end{picture}} at 3 4
\put{\begin{picture}(3,0)\boxes41\end{picture}} at 3 2
\put{$\times$} at 2.7 2
\put{\begin{picture}(3,0)\boxes33\end{picture}} at 4 6
\put{\begin{picture}(3,0)\boxes20\end{picture}} at 4 4.5
\put{{\bf (*)}} at 4 4
\put{\begin{picture}(3,0)\lboxes41\end{picture}} at 4 3
\put{\begin{picture}(3,0)\boxes51\end{picture}} at 4 1
\put{$\times$} at 3.7 1
\put{\begin{picture}(3,0)\lboxes42\end{picture}} at 5 4
\put{\begin{picture}(3,0)\lboxes51\end{picture}} at 5 2
\put{\begin{picture}(3,0)\boxes61\end{picture}} at 5 0
\put{$\times$} at 4.7 0
\put{\begin{picture}(3,0)\boxes30\end{picture}} at 6 6
\put{\begin{picture}(3,0)\boxes43\end{picture}} at 6 4.5
\put{\begin{picture}(3,0)\lboxes52\end{picture}} at 6 3
\put{\begin{picture}(3,0)\lboxes61\end{picture}} at 6 1
\put{\begin{picture}(3,0)\lboxes53\end{picture}} at 7 4
\put{\begin{picture}(3,0)\lboxes62\end{picture}} at 7 2
\put{\begin{picture}(3,0)\boxes53\end{picture}} at 8 6
\put{\begin{picture}(3,0)\boxes40\end{picture}} at 8 4.5
\put{\begin{picture}(3,0)\lboxes63\end{picture}} at 8 3
\put{\begin{picture}(3,0)\lboxes64\end{picture}} at 9 4
\put{\begin{picture}(3,0)\boxes63\end{picture}} at 10 4.5
\put{\begin{picture}(3,0)\boxes50\end{picture}} at 10 6
\arr{.3 5.4}{.7 4.6}
\arr{1.3 4.6}{1.7 5.4}
\arr{1.3 4.15}{1.7 4.35}
\arr{1.3 3.7}{1.7 3.3}
\arr{2.3 5.4}{2.7 4.6}
\arr{2.3 4.35}{2.7 4.15}
\arr{2.3 3.3}{2.7 3.7}
\arr{2.3 2.7}{2.7 2.3}
\arr{3.3 4.6}{3.7 5.4}
\arr{3.4 4.2}{3.7 4.35}
\arr{3.3 3.7}{3.7 3.3}
\arr{3.3 2.3}{3.7 2.7}
\arr{3.3 1.7}{3.7 1.3}
\arr{4.3 5.4}{4.7 4.6}
\arr{4.3 4.35}{4.7 4.15}
\arr{4.4 3.4}{4.7 3.7}
\arr{4.3 2.7}{4.7 2.3}
\arr{4.3 1.3}{4.7 1.7}
\arr{4.3 .7}{4.7 .3}
\arr{5.3 4.6}{5.7 5.4}
\arr{5.4 4.2}{5.7 4.35}
\arr{5.3 3.7}{5.7 3.3}
\arr{5.3 2.3}{5.7 2.7}
\arr{5.3 1.7}{5.7 1.3}
\arr{5.3 .3}{5.7 .7}
\arr{6.3 5.4}{6.7 4.6}
\arr{6.3 4.35}{6.7 4.15}
\arr{6.4 3.4}{6.7 3.7}
\arr{6.3 2.7}{6.7 2.3}
\arr{6.3 1.3}{6.7 1.7}
\arr{7.4 4.8}{7.7 5.4}
\arr{7.4 4.2}{7.7 4.35}
\arr{7.3 3.7}{7.7 3.3}
\arr{7.4 2.4}{7.7 2.7}
\arr{8.3 5.4}{8.7 4.6}
\arr{8.3 4.35}{8.7 4.15}
\arr{8.4 3.4}{8.7 3.7}
\arr{9.4 4.8}{9.7 5.4}
\arr{9.4 4.2}{9.7 4.35}
\setdots<2pt>
\plot 0.3 6  1.7 6 /
\plot 2.3 6  3.7 6 /
\plot 4.3 6  5.7 6 /
\plot 6.3 6  7.7 6 /
\plot 8.3 6  9.7 6 /
\plot 2.3 4.5  3.7 4.5 /
\plot 4.3 4.5  4.7 4.5 /
\plot 5.35 4.5  5.7 4.5 /
\plot 6.3 4.5  6.7 4.5 /
\plot 7.35 4.5  7.7 4.5 /
\plot 8.3 4.5  8.7 4.5 /
\plot 9.35 4.5  9.7 4.5 /
\multiput{$\ddots$} at 10 3  8 1 /
\endpicture
$$

In the diagram, the sequence ending at $T_2^{4,2}$ is labelled by $(*)$. 
We will visualize in (\ref{subsection-S5}), Example (2),
how this sequence determines all the indecomposable objects 
in the Auslander-Reiten quiver for $\mathcal S(5)$
which have a symbol $\singlebox{2_2}$ in the 4-th row of their Klein tableau.

\smallskip
For later use we record the sink maps ending at a picket of type $P_2^m$:
$$v_2^m \;:\;\left\{\begin{array}{r@{\to}l@{\quad\text{if}\quad}l}
         P_1^2 & P_2^2 & m=2 \\
         T_2^{m,m-2} & P_2^m & m\geq 3 \\
\end{array}\right.$$
and the sink maps ending at a bipicket of type $T_2^{m,r}$:
$$v_2^{m,r}\;:\;\left\{\begin{array}{r@{\to}l@{\quad\text{if}\quad}l}
        P_0^1\oplus P_2^2\oplus P_1^3 & T_2^{3,1} & m=3, r=1 \\
        P_0^{m-2}\oplus P_2^{m-1}\oplus T_2^{m,m-3} & T_2^{m,m-2} 
                                            & m\geq 4, r=m-2 \\
        P_1^m \oplus T_2^{m-1,1} & T_2^{m,1} & m\geq 4, r=1 \\
        T_2^{m,r-1}\oplus T_2^{m-1,r} & T_2^{m,r} & 2\leq r \leq m-3 \\
\end{array}\right.$$

To unify notation we also write $T_2^{m,m-1}=P_2^m$ and define 
$v_2^{m,m-1}=v_2^m$.

\section{Hall Polynomials}\label{section-Hall}

We assume throughout this section that the residue field 
$k=R/(p)$ is the finite field of $q$ elements.
The sum-product 
formula in \cite{klein2}, in conjunction with an application of the
orbit equation, yields the formula for the  Hall
polynomial in Theorem~\ref{theorem-Hall}.
In an example we illustrate how Klein tableaux control
the counting process.

\subsection{The action of the automorphism group of $B$}\label{subsection-orbit}

Let $(A\subset B)$ be an embedding in $\mathcal S$. The group $G=\Aut_RB$ acts on the set
$$\{(U\subset B)\in\mathcal S\;|\; U\subset B\}$$
via $\beta\cdot(U\subset B) = (\beta(U)\subset B)$.
The cardinality of the orbit of $(A\subset B)$ under this action is the 
Hall multiplicity of  the embedding
$$g(A\subset B) \;=\; \#\big\{ U\subset B \;\big|\; 
          (U\subset B)\cong (A\subset B) \;\text{in}\;\mathcal S\big\}$$
while the stabilizer of $(A\subset B)$ is the automorphism group $\Aut_{\mathcal S}(A\subset B)$.
From the orbit formula we obtain
$$g(A\subset B) \;=\; \frac{\#\Aut_RB}{\#\Aut_{\mathcal S}(A\subset B)}.$$

\subsection{Klein's sum-product formula}

We deduce Theorem~\ref{theorem-Hall} by applying the above to a sum-product formula
which is implicitely in Klein's original article on the computation 
of Hall polynomials \cite{klein2}.
We first exhibit the formula and then use (\ref{subsection-orbit}) to prove
Theorem~\ref{theorem-Hall}.

\smallskip
{\it Notation.} Let $\Pi$ be a Klein tableau of type $(\alpha,\beta,\gamma)$.
The Hall multiplicity of $\Pi$ in $R_p$-mod is denoted by
$$g(\Pi;q)\;=\; \#\big\{\;U\subset M(\beta)\;\big|\;\Pi((U\subset M(\beta))) = \Pi\;\big\}.$$

\begin{proposition}\label{proposition-Klein-Hall}
For partitions $\alpha,\beta,\gamma$ and $e=\alpha_1$, the Hall polynomial in $R_p$-mod
can be computed as 
$$g_{\alpha\gamma}^\beta(q)\;=\; \sum_\Pi \; \prod_{\ell=2}^{e+1}
        \;  \frac{ g(\Pi|_2^\ell;q)}{ g(\Pi|_1^\ell;q)}$$
where the sum is taken over all Klein tableaux of type $(\alpha,\beta,\gamma)$. 
\end{proposition}

\begin{proof}
According to \cite[Corollaries 1-3, p.~77]{klein2},
the Hall polynomial is computed as 
$$(*)\qquad g_{\alpha\gamma}^\beta(q) \;=\; q^{\|\beta\|-\|\alpha\|-\|\gamma\|}
           \sum_{\Pi}\prod_{\ell=2}^{e+1} F(\Pi|^\ell_2;\textstyle\frac1q),$$
where the sum is over all Klein tableaux of type $(\alpha,\beta,\gamma)$.  Here,
$F(\Pi,t)$ is a polynomial in $t$
given explicitely in \cite[(3.5)]{klein2}.
The moment of a partition $\mu$ is given as 
$\|\mu\|=\sum_{i\geq 1}{\mu_i' \choose 2}$, so in particular
$\|\alpha\|=\sum_\ell\|(1^{\alpha_\ell'})\|$, and we have
$$q^{\|\beta\|-\|\gamma\|-\|\alpha\|} \;=\; 
   \prod_{\ell\geq 1} q^{\|\gamma^\ell\|-\|\gamma^{\ell-1}\|-\|(1^{\alpha_\ell'})\|} \;=\;
   \prod_{\ell\geq2} q^{\|\gamma^{\ell-1}\|-\|\gamma^{\ell-2}\|-\|(1^{\alpha_{\ell-1}'})\|}.$$
Hence  the first factor in $(*)$ can be distributed over the factors in the product.
Let $U\subset B$ be an elementary subgroup such that the embedding $(U\subset B)$
has Klein tableau $\Pi|^\ell_1$. 
According to \cite[Theorem~3.7]{klein2} the number of subgroups $A$ in $B$
which have Klein tableau  $\Pi|_2^\ell$ and for which $pA=U$ holds is computed as
$$h(\Pi|_2^\ell,q)\;=\;
     q^{\|\gamma^{\ell-1}\|-\|\gamma^{\ell-2}\|-\|(1^v)\|} F(\Pi|_2^\ell,\textstyle\frac1q)$$
where $v=|\gamma^{\ell-1}\backslash\gamma^{\ell-2}|$. 
Since the number of elementary subgroups $U$ in $B$ of tableau $\Pi|^\ell_1$
is $g(\Pi|^\ell_1,q)$, and since all such embeddings $(U\subset B)$ are isomorphic
in $\mathcal S$, we can write $h(\Pi|_2^\ell;q)$ as a quotient of two Hall multiplicities,
as needed:
$$h(\Pi|_2^\ell,q) \;=\; \frac{g(\Pi|_2^\ell,q)}{g(\Pi|^\ell_1,q)}.$$
\end{proof}

\smallskip
We deduce the formula in Theorem~\ref{theorem-Hall}:

\begin{proof}[Proof of Theorem~\ref{theorem-Hall}]
The key point is that each Klein tableau with entries
at most 2 determines a unique isomorphism class of short exact sequences, so that 
we can apply the orbit formula in (\ref{subsection-orbit}).
In particular, we obtain for the embedding 
$E(\Pi|^\ell_2;q)=(A\subset B)$ corresponding to the Klein tableau $\Pi|^\ell_2$:
$$g(\Pi|^\ell_2;q)\;=\; g((A\subset B)) \;=\; \frac{\#\Aut_RB}{\#\Aut_{\mathcal S}(A\subset B)}$$
Hence we have
$$\frac{g(\Pi^\ell_2;q)}{g(\Pi^\ell_1;q)}\;=\; 
        \frac{\#\Aut_{\mathcal S}(pA\subset B)}{\#\Aut_{\mathcal S}(A\subset B)}
        \;=\; \frac{\#\Aut_{\mathcal S}E(\Pi|^\ell_1;q)}{\#\Aut_{\mathcal S}E(\Pi|^\ell_2;q)}.$$
It follows from Proposition~\ref{proposition-Klein-Hall} that 
$$g_{\alpha\gamma}^\beta (q) \;=\; 
   \sum_\Pi \prod_{\ell=2}^{e+1} \frac{\#\Aut_{\mathcal S} E(\Pi|^\ell_1;q)}{%
                                     \#\Aut_{\mathcal S} E(\Pi|^\ell_2;q)}.$$
\end{proof}

\subsection{Remarks}

\begin{enumerate}
\item Consider the corresponding sum-product formula for 
LR-tab\-leaux from \cite[II, (4.2) and (4.9)]{macdonald}:
$$g_{\alpha\gamma}^\beta(q)\;=\; 
     \sum_\Gamma \prod_{\ell=2}^{e+1} h(\Gamma|_2^\ell;q),$$
where the sum is over all LR-tableaux $\Gamma$ of type 
$(\alpha,\beta,\gamma)$. 
By $\Gamma|_i^\ell$ we denote the restriction of the tableau 
to the entries $\ell,\ell-1,\ldots, \ell-i$, as for Klein tableaux.
Again, the factors can be written as 
$$h(\Gamma|_2^\ell;q) \;=\; 
     \frac{g(\Gamma|_2^\ell;q)}{g(\Gamma|_1^\ell;q)},$$
where $g(\Gamma|_1^\ell;q)=g(\Pi|_1^\ell;q)$ counts elementary embeddings. 
However, the $p^2$-bounded embeddings counted by 
$g(\Gamma|_2^\ell;q)$ are not necessarily isomorphic.

\item For any LR-tableau $\Gamma$, the formula 
$g(\Gamma;q)=\sum_\Pi g(\Pi;q)$
holds where the sum is over all Klein tableaux $\Pi$ refining $\Gamma$.  
Note that 
each of the polynomials $g(\Gamma;q)$, $g(\Pi;q)$ is monic, so there is a unique
Klein tableau $\Pi_D$,
called in \cite{klein2} the {\it dominant Klein tableau corresponding to} $\Gamma$, which
refines $\Gamma$ and which is such that $g(\Pi_D;q)$ 
has the same degree as $g(\Gamma;q)$, and hence the same degree as the
Hall polynomial.

\item In the above sum-product formula 
$g_{\alpha\gamma}^\beta(q)=
        \sum_\Pi\prod_{\ell=2}^{e+1}h(\Pi|_2^\ell;q)$, 
the factors are indexed by uniquely determined isomorphism classes of
$p^2$-bounded embeddings.  
Note that one cannot expect a refined version of this formula where
the factors are described by uniquely determined isomorphism types of 
$p^3$-bounded embeddings since 
parametrized families occur already in the case where $B$ has exponent 7
\cite{bounded}.
\end{enumerate}

\subsection{Example}

In order to demonstrate that the formula in Theorem~\ref{theorem-Hall} is 
computationally feasible, we determine the Hall polynomial $g_{321,21}^{432}$.

\smallskip
The following lemma will provide us with the size of some homomorphism and 
automorphism groups.

\begin{lemma}
\begin{enumerate}
\item $\#\End_{\mathcal S} T_2^{m,r} = q^{m+3r-1}$, \newline
      $\#\Aut_{\mathcal S} T_2^{m,r} = (1-\frac1q) q^{m+3r-1}$
\item $\#\Hom_{\mathcal S} (P_u^v,P_\ell^m) = q^{\min\{m,v-\max\{0,u-\ell\}\}}$, \newline
      $\#\Aut_{\mathcal S} P_\ell^m = (1-\frac1q) q^m$
\item $\#\Hom_{\mathcal S} (T_2^{v,w},P_\ell^m) =
       \left\{\begin{array}{ll}  q^{\min\{v-1,m\}+\min\{w-1,m\}}, & \text{if}\; \ell=0 \\
                                 q^{\min\{v-1,m\}+\min\{w,m\}}, & \text{if}\; \ell=1 \\
                                 q^{\min\{v,m\}+\min\{w,m\}}, & \text{if}\; \ell\geq 2
              \end{array}\right.$
\item $\#\Hom_{\mathcal S} (P_u^v, T_2^{m,r}) =
       \left\{\begin{array}{ll}  q^{\min\{v,r\}+\min\{v,m\}}, & \text{if}\; u=0\\
                                 q^{\min\{v-1,r\}+\min\{v,m\}}, & \text{if}\; u=1\\
                                 q^{\min\{v-u+1,r\}+\min\{v-u+1,m\}}, & \text{if}\; u\geq 2
              \end{array}\right.$
\end{enumerate}
\end{lemma}

\begin{proof}
Corollary~\ref{corollary-hom2bipicket} yields the results stated in (1) and in (4).
For (2) and (3), use the equality
$$\Hom_{\mathcal S}( \,(A\subset B), P_\ell^m\,)\;=\; \Hom_R(\,B/p^\ell A, P^m\,)$$
discussed in \cite{lr} and after Corollary~\ref{corollary-hom2bipicket}.
\end{proof}

\begin{example}
For $\alpha=(321)$, $\beta=(432)$ and $\gamma=(21)$ as in (\ref{section-Klein-tableaux}) 
we compute the Hall polynomial
$g_{\alpha,\gamma}^\beta$ using the formula in Theorem~\ref{theorem-Hall}.
There are three Klein tableaux of type $(\alpha,\beta,\gamma)$:

$$
\setcounter{boxsize}{3}
\beginpicture\setcoordinatesystem units <1cm,1cm>
\put {} at 0 .5
\put {} at 7 .5
\put{$\Pi_1:$} at -1 0
\put{\begin{picture}(9,12)\put(0,6){\boxes43}
                          \put(3,9){\boxes21}\put(3,3){\numbox 3}
                          \put(6,9){\boxes23}
     \end{picture}} at 0 0
\put{$\Pi=\Pi_2:$} at 2 0
\put{\begin{picture}(9,12)\put(0,7){\boxes31}\put(0,0){\numbox{3_2}}
                          \put(3,7){\boxes33}
                          \put(6,9){\boxes23}
     \end{picture}} at 3.5 0
\put{$\Pi_3:$} at 5 0
\put{\begin{picture}(9,12)\put(0,7){\boxes31}\put(0,0){\numbox {3_3}}
                          \put(3,7){\boxes33}
                          \put(6,9){\boxes23}
     \end{picture}} at 6 0
\endpicture
$$

The sum in the formula is indexed by the Klein tableaux of the given type:
$$ g_{\alpha,\gamma}^\beta \;=\; g(\Pi_1) + g(\Pi_2) + g(\Pi_3)$$
We put $\Pi=\Pi_2$ and compute $g(\Pi)$ first:
$$g(\Pi) = \frac{ \#\Aut\mathcal E(\Pi|^2_1) }{ \#\Aut\mathcal E(\Pi|^2_2) } \cdot
           \frac{ \#\Aut\mathcal E(\Pi|^3_1) }{ \#\Aut\mathcal E(\Pi|^3_2) } \cdot
           \frac{ \#\Aut\mathcal E(\Pi|^4_1) }{ \#\Aut\mathcal E(\Pi|^4_2) }$$
The restrictions of $\Pi$,
$$
\setcounter{boxsize}{3}
\beginpicture\setcoordinatesystem units <1cm,1cm>
\put {} at 0 .5
\put {} at 7 .5
\put{$\Pi|_2^2:$} at 0 0
\put{\begin{picture}(9,12)\put(0,7){\boxes31}
                          \put(3,7){\boxes33}
                          \put(6,9){\boxes23}
     \end{picture}} at 1 0
\put{$\Pi|_1^2:$} at 2.5 0
\put{\begin{picture}(9,12)\put(0,7){\boxes30}
                          \put(3,7){\boxes31}
                          \put(6,9){\boxes21}
     \end{picture}} at 3.5 0
\put{$\Pi|_2^3:$} at 5 0
\put{\begin{picture}(9,12)\put(0,7){\boxes30}\put(0,0){\numbox {2_2}}
                          \put(3,7){\boxes31}
                          \put(6,9){\boxes21}
     \end{picture}} at 6 0
\put{$\Pi|_1^3=\Pi|_2^4:$} at 7.8 0
\put{\begin{picture}(9,12)\put(0,6){\boxes41}
                          \put(3,7){\boxes30}
                          \put(6,9){\boxes20}
     \end{picture}} at 9.4 0
\put{$\Pi|_1^4:$} at 10.55 0
\put{\begin{picture}(9,12)\put(0,6){\boxes40}
                          \put(3,7){\boxes30}
                          \put(6,9){\boxes20}
     \end{picture}} at 11.5 0
\endpicture
$$
define the following direct sum decompositions of the short exact sequences
corresponding to the tableaux (Proposition~\ref{proposition-KRS-multiplicity}):
\begin{eqnarray*}g(\Pi) & = & 
  \frac{\#\Aut(P_0^3\oplus P_1^3\oplus P_1^2) }{ \#\Aut(P_1^3\oplus P_2^3\oplus P_2^2) } 
  \;\cdot\;
  \frac{\#\Aut(P_1^4\oplus P_0^3\oplus P_0^2) }{ \#\Aut(T_2^{42}\oplus P_1^3) } \\
  &  & \qquad \;\cdot\;
  \frac{\#\Aut(P_0^4\oplus P_0^3\oplus P_0^2) }{ \#\Aut(P_1^4\oplus P_0^3 \oplus P_0^2)}
  \end{eqnarray*}

For the first factor on the right hand side, note that the modules in the numerator and 
in the denominator correspond to each other under the functors $\uparrow$, $\downarrow$
defined in (\ref{subsection-lifting}).  The adjointness isomorphism in
Lemma~\ref{lemma-adjoint} yields an isomorphism between the automorphism groups.
Hence the first factor is 1.

Also the third factor is 1 since the total spaces of the direct sums in numerator and
denominator agree while the subgroups are invariant under automorphisms.

We determine the second factor using the lemma above. 
\begin{eqnarray*}
\#\Aut_{\mathcal S}(P_1^4\oplus P_0^3\oplus P_0^2) & = & 
       (1-\textstyle\frac1q)^3 \;q^{23}\\
\#\Aut_{\mathcal S}(T_2^{42}\oplus P_1^3) & = & 
      (1-\textstyle\frac1q)^2\; q^{22}
\end{eqnarray*}

Hence $g(\Pi_2)=q-1$. The computation of $g(\Pi_1)$ and $g(\Pi_3)$ turns out to 
be even less involved since no bipickets occur.  The numbers are $g(\Pi_1)=q^2=g(\Pi_3)$.

\smallskip
In conclusion, $g_{321,21}^{432}(q)=g(\Pi_1)+g(\Pi_2)+g(\Pi_3)=2q^2+q-1$.
\end{example}

\section{Homomorphisms in $\mathcal S$}\label{section-categorification}

In this section we interpret the entries in the Klein tableau in terms of
homomorphisms in the category $\mathcal S$.

\subsection{Lifting and Reducing}\label{subsection-lifting}

We consider the following two endofunctors on $\mathcal S$:

\begin{eqnarray*}
\up: & (A\subset B)\longmapsto & 
              (p^{-1}A\subset B) \qquad \text{``lifting''}\\
\down: & (A\subset B)\longmapsto & 
              (pA\subset B) \qquad \text{``reducing''}
\end{eqnarray*}

Here $p^{-1}A$ denotes the subgroup $\{b\in B\,|\,pb\in A\}$. 
For $s$ a nonnegative integer, $\up^s$, $\down_s$ 
are $s$ iterations of $\up$ and $\down$, respectively.

\smallskip
The following corollary is an immediate consequence of the formulas
$p^{-1}pA = A+\soc B$, $pp^{-1}A = A\cap \rad B$, $pp^{-1}pA=pA$, and
$p^{-1}pp^{-1}A = p^{-1}A$.

\begin{corollary}\label{corollary-updownup}
Let $(A\subset B)\in \mathcal S$.  Then $A\subset \rad B$ if and only if
$(A\subset B)\up\down= (A\subset B)$.  Also, $\soc B\subset A$ if and 
only if $(A\subset B)\down\up=(A\subset B)$.  Moreover,
$$(A\subset B)\up\down\up = (A\subset B)\up \quad\text{and}\quad
   (A\subset B)\down\up\down = (A\subset B) \down.$$\qed
\end{corollary}

\begin{example}
For $P_\ell^m$ a picket, and $s\geq 0$, we have
$$P_\ell^m\up^s = P_{u}^m \qquad\text{where}\qquad u=\min\{\ell+s,m\}.$$

For $T_2^{m,r}$ a bipicket, and $0\leq s\leq r-1$, we obtain the embedding
$$T_2^{m,r}\up^s 
  \;=\; \big(\;( (p^{m-2-s},p^{r-1-s}),(0,p^{r-s}))\;\subset\;P^m\oplus P^r \;\big)$$
where the subgroup is isomorphic to $P^{2+s}\oplus P^s$, and the factor to
$P^{m-s-1}\oplus P^{r-s-1}$.
In this case, $T_2^{m,r}\up^s\down_s = T_2^{m,r}$.
However, if $s> r-1$ then the pair $(A\subset B)=T_2^{m,r}\up^s$ 
is the direct sum of pickets
$$T_2^{m,r}\up^s = P^m_u \oplus P^r_r \quad\text{where}\quad
                                           u=\min\{1+s,m\}.$$
(To see this, note that for a pair $(A\subset B)$ with $B'$ 
a direct summand of $B$ contained in $A$,
the pair $(B'=B')$ is a direct summand
of $(A\subset B)$.  In this case, take $B'=P^r$.)

Here are some Klein tableaux of raised embeddings.
$$\beginpicture\setcoordinatesystem units <1.2cm,1.2cm>
  \setcounter{boxsize}{3}
  \put {} at 0 .5
  \put {} at 0 -.6
  \put {$P_2^5:$} at .2 0 
  \put {\begin{picture}(3,7)
       \put(0,0){\boxes53}
       \end{picture}} at 1 .5
  \put {$P_2^5\up^2:$} at 2 0
  \put {\begin{picture}(3,7)
       \put(0,3){\boxes33}\put(0,-4){\numbox 3}\put(0,-7){\numbox 4}
       \end{picture}} at 3 .5  
  \put {$T_2^{5,3}:$} at 4.5 0
  \put {\begin{picture}(6,7)
       \put(0,0){\lboxes 53}
       \end{picture}} at 5.5 .5
  \put {$T_2^{5,3}\up^2:$} at 7 0
  \put {\begin{picture}(6,7)
       \put(0,2){\boxes43}\put(3,5){\boxes23}
       \put(3,-1){\numbox 3}\put(0,-7){\numbox 4}
       \end{picture}} at 8.1 .5
\endpicture$$

\end{example}

The following lemma 
is an immediate consequence of the definition:

\begin{lemma}\label{lemma-adjoint}
The functors $\down_s, \up^s$ form an adjoint pair.
More precisely, we have for
$E$, $F$ in $\mathcal S$:
$$\Hom_{\mathcal S}(E\down_s, F)\;=\; 
                      \Hom_{\mathcal S}(E,F\up^s).$$
\qed
\end{lemma}

\begin{lemma}\label{lemma-reduced-prototype}
Suppose the Klein tableau for $E\in\mathcal S$ is given by 
$\Pi=[\gamma^0,\ldots,\gamma^e;\varphi^2,\ldots,\varphi^e]$, and $s\leq e$.
Then $E\!\down_s$ has Klein tableau 
$$\Pi(E\!\down_s)\;=\;
        [\gamma^s,\ldots,\gamma^e;\varphi^{2+s},\ldots,\varphi^e] \;=\; \Pi|_{e-s}^e.$$
\end{lemma}

\begin{proof}
Suppose $E$ is given by the embedding $(A\subset B)$ and its Klein tableau 
$\Pi$ is defined by the partition sequence $(\gamma^{\ell,r})_{\ell,r}$ where
$\gamma^{\ell,r}=\type(B/p^\ell A+p(p^{\ell-2}A\cap p^r B))$.
Then $E\!\down_s=(p^sA\subset B)$, hence the Klein tableau
$\Pi(E\!\down_s)$ is defined by the partition sequence 
$(\gamma^{\ell+s,r})_{\ell,r}$.  
\end{proof}

\subsection{The categories $\mathcal S_2\up^{\ell-2}$}\label{section-s2up}

For $\ell\geq 2$, the category $\mathcal S_2\up^{\ell-2}$ 
consists of all pairs $(A\subset B)$
in $\mathcal S_\ell$ where $\soc^{\ell-2} B\subset A$.
For each indecomposable pair not in $\mathcal S_{\ell-1}$ we determine the
sink map in $\mathcal S_2\up^{\ell-2}$ and use this information to picture
the partial Auslander-Reiten quiver.

\begin{definition}
For $\mathcal U$ one of the categories $\mathcal S$, $\mathcal S_\ell$,
$\mathcal S(n)$, or $\mathcal S_\ell(n)$, 
let $\mathcal U\up$ and $\mathcal U\down$
be the full subcategories of $\mathcal S$ of all pairs $(A\subset B)\up$
and $(A\subset B)\down$, respectively, where $(A\subset B)\in \mathcal U$.
\end{definition}

\begin{lemma}\label{lemma-catiso}
The functors $\downarrow$, $\uparrow$ induce categorical isomorphisms
$$\mathcal S\up \cong \mathcal S\down,   \quad
   \mathcal S(n)\up \cong \mathcal S(n)\down,  \quad
   \mathcal S_\ell\up \cong\mathcal S_{\ell+1}\down,   \quad
   \mathcal S_\ell(n)\up  \cong \mathcal S_{\ell+1}(n)\down.$$
\end{lemma}

\begin{proof}
According to Corollary~\ref{corollary-updownup}, 
the functors $\downarrow\up$, $\uparrow\down$ are the identity
on each object and on each homomorphism group.
\end{proof}

\begin{proposition}\label{proposition-sink}
Let $\ell>2$. If $(A\subset B)\up^{\ell-2}$ in $\mathcal S_2\up^{\ell-2}$ 
is an indecomposable object with $p^{\ell-1}A\neq 0$, 
and $v$ the sink map for $(A\subset B)$ in $\mathcal S_2$,
then the sink map $v'$ for $(A\subset B)\up^{\ell-2}$ 
in $\mathcal S_2\up^{\ell-2}$
is the minimal version of $v\up^{\ell-2}$.
In particular, for $E\in\mathcal S$, 
$$\Im\Hom_{\mathcal S}(E,v')=\Im\Hom_{\mathcal S}(E,v\up^{\ell-2}).$$
\end{proposition}

\begin{proof}
By Lemma~\ref{lemma-catiso}, the category $\mathcal S_2\up^{\ell-2}$ is equivalent
to the full subcategory of $\mathcal S_2$ of all pairs $(A\subset B)$ where
$A\subset p^{\ell-2}B$.  
We obtain the sink maps in the proposition by taking minimal versions of the
liftings of the corresponding sink maps in $\mathcal S_2$. 
We are interested in pickets of the form $P_\ell^m$ where $\ell\leq m$ (cases (a) and (b)),
and in bipickets of type $T_2^{m,r}\up^{\ell-2}$ where $\ell\leq r+1< m$ (cases (c) and (d)).

\begin{itemize}
\item[(a)] The picket $P_\ell^\ell$ is projective in $\mathcal S_\ell$, and hence
in $\mathcal S_2\up^{\ell-2}$, and has as sink map the inclusion
$$ P_{\ell-1}^\ell \to P_\ell^\ell.$$
This is the minimal version of the map
$v\up^{\ell-2}: P_{\ell-1}^\ell\oplus P_{\ell-2}^{\ell-2}\to P_\ell^\ell$,
obtained by lifting the sink map $v:T_2^{\ell,\ell-2}\to P_2^\ell$ 
in $\mathcal S_2$ (see the example in (\ref{subsection-lifting})).

\item[(b)] For $m>\ell$, the Auslander-Reiten sequence for the picket 
$P_2^m$ in $\mathcal S_2$,
$$0\to P_0^{m-2} \to T_2^{m,m-2} \to P_2^m \to 0 $$
yields the Auslander-Reiten sequence for the picket $P_\ell^m$ in
$\mathcal S_2\up^{\ell-2}$,
$$0 \to P_{\ell-2}^{m-2} 
    \to T_2^{m,m-2}\up^{\ell-2} 
    \to P_\ell^m \to 0.$$

\item[(c$_1$)]  We first consider the case where $\ell=r+1<m-1$.
The first term of the sink map in $\mathcal S_2$,
$$ T_2^{m-1,r}\oplus T_2^{m,r-1} \to T_2^{m,r}$$
decomposes after lifting into $\mathcal S_2\up^{\ell-2}$;
the lifted map is a right almost split morphism in this category:
$$ T_2^{m-1,r}\up^{\ell-2} \oplus P_{\ell-1}^m \oplus P_{r-1}^{r-1} 
       \to T_2^{m,r}\up^{\ell-2}$$
The minimal version of this map is the sink map $v$ for $T_2^{m,r}\up^{\ell-2}$
in the Auslander-Reiten sequence in $\mathcal S_2\up^{\ell-2}$,
$$ 0   \to P_{\ell-1}^{m-1}
       \to T_2^{m-1,r}\up^{\ell-2} \oplus P_{\ell-1}^m 
       \stackrel v\to T_2^{m,r}\up^{\ell-2}
       \to 0. $$

\item[(c$_2$)] Similarly in case $\ell=r+1=m-1$, one shows that the sink map
for  $T_2^{m,r}\up^{\ell-2}$ in $\mathcal S_2\up^{\ell-2}$ is the
epimorphism in the sequence:
$$  0   \to P_{\ell-1}^{\ell}
       \to P_\ell^\ell \oplus P_{\ell-2}^{\ell-1} \oplus P_{\ell-1}^{\ell+1} 
       \stackrel v\to T_2^{\ell+1,\ell-1}\up^{\ell-2}
       \to 0. $$

\item[(d)] Finally for $\ell< r+1 < m$, 
the Auslander-Reiten sequence in $\mathcal S_2$ ending in $T_2^{m,r}$ 
yields an Auslander-Reiten sequence in $\mathcal S_2\up^{\ell-2}$ of type
$$\qquad 0 \to T_2^{m-1,r-1}\up^{\ell-2}  
    \to T_2^{m,r-1}\up^{\ell-2}  \oplus T_2^{m-1,r}\up^{\ell-2}  
    \to T_2^{m,r}\up^{\ell-2} 
    \to 0 $$
if $m > r+2$, and in case if $m=r+2$ of type:
$$\qquad 0 \to T_2^{m-1,r-1}\up^{\ell-2}  
    \to T_2^{m,r-1}\up^{\ell-2}  \oplus P_\ell^{m-1}\oplus P_{\ell-2}^r
    \to T_2^{m,r}\up^{\ell-2} 
    \to 0 $$
\end{itemize}
\end{proof}

There is no assertion in the proposition about sink maps ending at pickets
of type $P_0^m\up^{\ell-2}$ or $P_1^m\up^{\ell-2}$.  In fact,
there exist sink maps in $\mathcal S_2\up^{\ell-2}$ for the pickets of the form
$P_0^m\up^{\ell-2}$.  
As those maps will not be needed in the following, we leave it as an exercise for the reader
to determine the corresponding Auslander-Reiten sequences.
On the other hand, pickets of the form $P_1\up^{\ell-2}$ are neither projective nor
will they admit a sink map.
In the example below, they are labeled with an $\times$.

\begin{example}
Here is the partial Auslander-Reiten quiver for $\mathcal S_2\up^2$:

$$
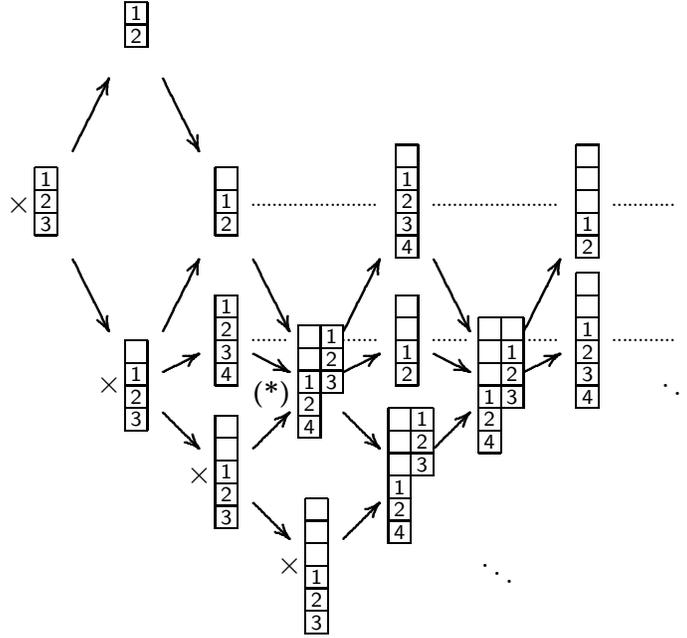
\beginpicture\setcoordinatesystem units <1.2cm,1.2cm>
\setcounter{boxsize}{3}
\put {} at 6 1.2
\put{\begin{picture}(3,0)\rrboxes31\end{picture}} at 4 6
\put{$\times$} at 3.7 6
\put{\begin{picture}(3,0)\boxes23\end{picture}} at 5 8
\put{\begin{picture}(3,0)\rrboxes41\end{picture}} at 5 4
\put{$\times$} at 4.7 4
\put{\begin{picture}(3,0)\boxes33\end{picture}} at 6 6
\put{\begin{picture}(3,0)\rrboxes43\end{picture}} at 6 4.5
\put{\begin{picture}(3,0)\rrboxes51\end{picture}} at 6 3
\put{$\times$} at 5.7 3
\put{\begin{picture}(3,0)\rrlboxes53\end{picture}} at 7 4
\put{{\footnotesize\bf(*)}} at 6.5 3.9
\put{\begin{picture}(3,0)\rrboxes61\end{picture}} at 7 2
\put{$\times$} at 6.7 2
\put{\begin{picture}(3,0)\rrboxes53\end{picture}} at 8 6
\put{\begin{picture}(3,0)\boxes43\end{picture}} at 8 4.5
\put{\begin{picture}(3,0)\rrlboxes63\end{picture}} at 8 3
\put{\begin{picture}(3,0)\rrlboxes64\end{picture}} at 9 4
\put{\begin{picture}(3,0)\rrboxes63\end{picture}} at 10 4.5
\put{\begin{picture}(3,0)\boxes53\end{picture}} at 10 6
\arr{4.3 6.6}{4.7 7.4}
\arr{5.3 7.4}{5.7 6.6}
\arr{4.3 5.4}{4.7 4.6}
\arr{5.3 4.6}{5.7 5.4}
\arr{5.3 4.15}{5.7 4.35}
\arr{5.3 3.7}{5.7 3.3}
\arr{6.3 5.4}{6.7 4.6}
\arr{6.3 4.35}{6.7 4.15}
\arr{6.3 3.3}{6.7 3.7}
\arr{6.3 2.7}{6.7 2.3}
\arr{7.3 4.6}{7.7 5.4}
\arr{7.3 4.15}{7.7 4.35}
\arr{7.3 3.7}{7.7 3.3}
\arr{7.3 2.3}{7.7 2.7}
\arr{8.3 5.4}{8.7 4.6}
\arr{8.3 4.35}{8.7 4.15}
\arr{8.3 3.3}{8.7 3.7}
\arr{9.3 4.6}{9.7 5.4}
\arr{9.3 4.15}{9.7 4.35}
\setdots<2pt>
\plot 6.3 6  7.7 6 /
\plot 8.3 6  9.7 6 /
\plot 10.3 6  11 6 /
\plot 6.3 4.5  6.7 4.5 /
\plot 7.3 4.5  7.7 4.5 /
\plot 8.3 4.5  8.7 4.5 /
\plot 9.3 4.5  9.7 4.5 /
\plot 10.3 4.5  11 4.5 /
\multiput{$\ddots$} at 11 4  9 2 /
\endpicture
$$
\end{example}

In Example (3) in (\ref{subsection-S5}) 
we will use the Auslander-Reiten sequence ending at $T_2^{5,3}\up^2$
and labelled (*) to determine all objects with a $\singlebox{4_3}$ in the 5-th row 
of their Klein tableau.

\subsection{Approximations}

For $E:(A\subset B)$ an embedding of $p$-modules, 
and $\ell$ a natural number, put
$$E|^\ell \;=\; \big( A/p^\ell A \subset B/p^\ell A\big).$$

\smallskip
The following results follow immediately from the  definition.

\begin{lemma}\label{lemma-approximation}
The canonical map $\pi:E\to E|^\ell$ is a 
minimal left approximation of $E$ in $\mathcal S_\ell$.
\end{lemma}

\begin{proof}
Every map $E\to F$ with $F\in\mathcal S_\ell$ factors over $\pi$,
so $\pi$ is a left approximation for $E$ in $\mathcal S_\ell$.
Since $\pi$ is onto, any map $u\in\End(E|^\ell)$ which 
satisfies $\pi=u\pi$ is an isomorphism, so $\pi$ is in addition left
minimal.
\end{proof}

\begin{lemma}
\sloppypar{Suppose an embedding $E\in\mathcal S$ has Klein tableau $\Pi=[\gamma^0,\ldots,\gamma^e;
\varphi^2,\ldots,\varphi^e]$, and $\ell$ is a natural number at most $e$.
Then}
$$\Pi(E|^\ell)\;=
\;[\gamma^0,\ldots,\gamma^\ell;\varphi^2,\ldots,\varphi^\ell] \;=\;\Pi|^\ell.$$\qed
\end{lemma}

We obtain from Lemma~\ref{lemma-reduced-prototype}:

\begin{lemma}
Suppose $E$ has Klein tableau $\Pi$ and $k\leq \ell$. 
The embeddings $E\down_{\ell-k}|^k$ and $E|^\ell\down_{\ell-k}$
coincide as objects in $\mathcal S$ and have Klein tableau
$$\Pi(E\down_{\ell-k}|^k)=
  [\gamma^{\ell-k},\gamma^{\ell-k+1},\ldots,\gamma^\ell;
    \varphi^{\ell-k+2},\varphi^{\ell-k+3},\ldots,\varphi^\ell]
  = \Pi|_k^\ell.
$$
In the case where $k=2$ we obtain
$$\Pi(E\down_{\ell-2}|^2)=
  [\gamma^{\ell-2},\gamma^{\ell-1},\gamma^\ell;\varphi^\ell]
  = \Pi|_2^\ell.
$$\qed
\end{lemma}

We write $E|_k^\ell=E|^\ell\down_{\ell-k}=E\down_{\ell-k}|^k$.
Combining this result with Proposition~\ref{proposition-KRS-multiplicity},
we can interpret the number of boxes $\singlebox{\ell_r}$ in the Klein tableau
of an embedding $E:(A\subset B)$ in terms of multiplicities of summands
in a subfactor of $E$.

\begin{corollary}\label{corollary-multiplicity}
Let $\Pi$ be the Klein tableau of an object $E\in\mathcal S$, and let
$2\leq\ell\leq r+1\leq m$ be positive integers.  
The number of boxes $\singlebox{\ell_r}$ in the $m$-th row of $\Pi$
equals the multiplicity of the picket $P_2^m$ (if $r=m-1$) 
or the bipicket $T_2^{m,r}$ (if $r<m-1$) in the 
direct sum decomposition for 
$$E|^\ell_2 \quad = \quad 
(p^{\ell-2}A/p^\ell A\;\subset\;B/p^\ell A).$$
\end{corollary}

\begin{proof}
The number of boxes $\singlebox{\ell_r}$ in the $m$-th row of $\Pi$
equals the number of boxes $\singlebox{2_r}$ in the $m$-th row
of $\Pi|_2^\ell$.
According to Proposition~\ref{proposition-KRS-multiplicity}, for each
symbol $\singlebox{2_r}$ there is an indecomposable summand of type
$P_2^m$ or $T_2^{m,r}$ in the direct sum decomposition of the 
corresponding object $E|^\ell_2$ in $\mathcal S_2$.
\end{proof}

We recover the corresponding result for LR-tableaux \cite[Corollary~2]{lr}
in the case where $\ell > 1$:

\begin{corollary}
Let $\Gamma$ be the LR-tableau corresponding to an object $E\in\mathcal S$,
and let $1\leq \ell\leq m$.
The number of boxes $\singlebox \ell$ in the $m$-th row of $\Gamma$
equals the multiplicity of the picket $P_1^m$ in the direct sum
decomposition for $E|^\ell_1$.
\end{corollary}

\begin{proof}
Let $\ell > 1$.  Since $E|^\ell_1=E|^\ell_2\down_1$
and since
$$T_2^{m,r}\down_1\;=\;\left\{\begin{array}{ll} P_1^m & \text{if}\;r=m-1 \\
                         P_1^m\oplus P_0^r & \text{if}\;r<m\end{array}\right.$$
the multiplicity of $P_1^m$ as a direct summand of $E|^\ell_1$
is the sum of the multiplicities of $T_2^{m,r}$ in $E|^\ell_2$
for $r=1,\ldots,m-1$.
\end{proof}

\subsection{Categorification}

Given an object $E\in\mathcal S$, we can interprete the entries in the 
Klein tableau
for $E$ in terms of homomorphisms in $\mathcal S$ 
and in terms of the decomposition of subfactors of $E$.

\smallskip
We restate 
Theorem~\ref{theorem-prototypes} from the introduction.
The numbers $m,r,\ell$ are chosen in such a way that
all bipickets $T_2^{m,r}\up^{\ell-2}$ and all pickets $P_\ell^m$ with $\ell\geq 2$
are included.

\begin{theorem}\label{theorem-proto2}
  For $E\in\mathcal S$ and $2\leq\ell\leq r+1\leq m$, the following numbers 
  are equal.
  \begin{enumerate}
  \item The number of boxes $\singlebox{\ell_r}$  in the $m$-th row 
    in the Klein tableau for $E$.
  \item The multiplicity of $T_2^{m,r}$ as a direct summand of 
    $$E|^\ell_2=
           \big(p^{\ell-2}A/p^\ell A\subset B/p^\ell A\big).$$
  \item The $k$-dimension of
    $$\frac{\Hom_{\mathcal S}(E,T_2^{m,r}\!\up^{\ell-2})}{%
            \Im\Hom_{\mathcal S}(E,v_2^{m,r}\!\up^{\ell-2})}.$$
  \end{enumerate}
\end{theorem}

\begin{remark}
\begin{enumerate}
\item In the theorem, $v_2^{m,r}$ is the sink map for $T_2^{m,r}$ in the category 
$\mathcal S_2$ as introduced in (\ref{section-AR-sequences}).
The lifting $v_2^{m,r}\up^{\ell-2}$ may not be the sink map 
for $T_2^{m,r}\up^{\ell-2}$, but its minimal version $v$ is 
(Proposition~\ref{proposition-sink}).  In particular,
$\Im\Hom(E,v_2^{m,r}\up^{\ell-2})=\Im\Hom(E,v)$ consists of all
maps which factor over the sink map for $T_2^{m,r}\up^{\ell-2}$
in the category $\mathcal S_2\up^{\ell-2}$. 
\item
The above theorem covers all entries in the Klein tableau with the exception 
of the $\singlebox1$'s. 
Those entries occur also in the underlying LR-tableau and are dealt with
by \cite[Theorem~1]{lr}:
The multiplicity of $\singlebox \ell$ in the $m$-th row of the LR-tableau
for $E$ equals the multiplicity of $P_1^m$ as a direct summand of 
$E|^\ell_1$ and also equals the $k$-dimension of 
$$\frac{\Hom(E,P_\ell^m)}{\Im\Hom(E,u_\ell^m)}$$
where $u_1^m$ is the sink map for $P_1^m$ in the category $\mathcal S_1$
and $u_\ell^m=u_1^m\up^{\ell-1}$.
\end{enumerate}
\end{remark}

\begin{proof}
According to Corollary~\ref{corollary-multiplicity}, 
the number of boxes labelled 
$\singlebox{\ell_r}$ in row $m$ is equal to the multiplicity of $T_2^{m,r}$
as a direct summand of $E|_2^\ell=E\!\down_{\ell-2}|^2$. 
In the category $\mathcal S_2$, this multiplicity is measured as the
dimension of the contravariant defect given by the Auslander-Reiten sequence
ending at $T_2^{m,r}$.  This dimension is equal to
$$\dim_k \;\frac{\Hom(E\!\down_{\ell-2}|^2,T_2^{m,r})}{%
                   \Im\Hom(E\!\down_{\ell-2}|^2,v_2^{m,r})}.$$
Using Lemma~\ref{lemma-approximation}, this number is equal to
$$\dim_k \;\frac{\Hom(E\!\down_{\ell-2},T_2^{m,r})}{%
                   \Im\Hom(E\!\down_{\ell-2},v_2^{m,r})}.$$
Now we apply the adjoint isomorphism from Lemma~\ref{lemma-adjoint} to obtain
equality with the expression in the theorem:
$$\dim_k \;\frac{\Hom(E,T_2^{m,r}\!\up^{\ell-2})}{%
                   \Im\Hom(E,v_2^{m,r}\!\up^{\ell-2})}.$$
\end{proof}

As a consequence of the theorem, we can read off from the Klein tableau
of an embedding  the length of the module of homomorphisms into a bipicket.
The key step is the reduction to the corresponding situation for pickets:

\begin{corollary}\label{corollary-hom2bipicket}
Suppose $E\in\mathcal S$ has Klein tableau $\Pi$. For integers $1\leq r\leq m-2$, let
$$b\;=\;\#\{\;\text{symbols $\singlebox{2_u}$ in row $v$ in $\Pi$} \;|\; 
                              r+2\leq v\leq m, 1\leq u\leq r \;\}$$
Then
\begin{eqnarray*}
\len\Hom_{\mathcal S}(E,T_2^{m,r}) & = &
         b+\len\Hom_{\mathcal S}(E,P_2^{r+1}\oplus P_0^r\oplus P_1^m)\\
                            & & \qquad    -\len \Hom_{\mathcal S}(E,P_1^{r+1}).
\end{eqnarray*}
\end{corollary}

The length of the homomorphism group into a picket $P_\ell^m$ can be read off
from the LR-tableau $\Gamma=[\gamma^0,\ldots,\gamma^e]$ of the module $E: (A\subset B)$:
$$\len\Hom_{\mathcal S}(E,P_\ell^m) \;=\;\sum_{i=1}^m (\gamma^\ell)'_i$$
(For this recall that 
$\Hom_{\mathcal S}(E,P_\ell^m)=\Hom_R(B/p^\ell A,P^m)$ as discussed in \cite{lr},
and note that $B/p^\ell A=\bigoplus_j P^{\gamma^\ell_j}$ 
and that $\len\Hom_R(P^s,P^m)=\min\{s,m\}$.)

\begin{proof}[Proof of the Corollary]
Since $\Hom(E,F)=\Hom(E|^2,F)$ for $F\in\mathcal S_2$ we may assume
that $E\in\mathcal S_2$. 
We first consider the case where 
$E$ has no summand of type $T_2^{v,u}$; then the functor $\Hom_{\mathcal S}(E,-)$
is exact when applied to the Auslander-Reiten sequence ending at $T_2^{v,u}$.

\smallskip
In the computation below 
we proceed along the diagonals in the Aus\-lan\-der-Rei\-ten quiver for $\mathcal S_2$,
going first from $T_2^{m,r}$ to $T_2^{m,1}$, then from $T_2^{m-1,r}$ to $T_2^{m-1,1}$,
etc.\ and finally from $T_2^{r+2,r}$ to $T_2^{r+2,1}$. 
In each step we replace the term $(E,T_2^{v,u})$, which is the length of the 
third term of the exact sequence
\begin{eqnarray*} 
0\to \Hom(E,X) & \to & \Hom(E,Y)\stackrel{\Hom(E,g^{v,u})}{\longrightarrow}\Hom(E,T_2^{v,u})\\
               & \to & \Cok\Hom(E,g^{v,u})\to 0
\end{eqnarray*}
given by applying the functor $\Hom(E,-)$ to the Auslander-Reiten sequence 
ending at $T_2^{v,u}$ by 
$$\len\Hom(E,Y)-\len\Hom(E,X).$$

\begin{eqnarray*}
(E,T_2^{m,r}) & = & (E,T_2^{m,r-1})+(E,T_2^{m-1,r})-(E,T_2^{m-1,r-1}) \\
 & = & (E,T_2^{m,r-2})+(E,T_2^{m-1,r})-(E,T_2^{m-1,r-2})\\
 & = & \cdots \; = \; (E,T_2^{m,1})+(E,T_2^{m-1,r})-(E,T_2^{m-1,1})\\
 & = & (E,P_1^m)+(E,T_2^{m-1,r}) -(E,P_1^{m-1}) \\
 & = & \cdots \; = \; (E,P_1^m)+(E,T_2^{r+2,r}) - (E,P_1^{r+2}) \\
 & = & (E,P_1^m)+(E,P_2^{r+1})+(E,P_0^r)+(E,T_2^{r+2,r-1})\\
 &   & \qquad -(E,T_2^{r+1,r-1}) - (E,P_1^{r+2})\\
 & = & \cdots \; = \; (E,P_1^m)+(E,P_2^{r+1})+(E,P_0^r)-(E,P_1^{r+1})
\end{eqnarray*}
Here $(E,Y)$ denotes the length $\len\Hom_{\mathcal S}(X,Y)$,
This shows the claim in case $b=0$.

\smallskip
Let us now consider an arbitrary embedding $E$.
Then for a given pair $(v,u)$, the above 
replacement of $(E,T_2^{v,u})$ by $(E,Y)-(E,X)$ omits the term $\len\Cok(E,g^{v,u})$ 
which according to Theorem~\ref{theorem-prototypes}
counts the number of boxes $\singlebox{2_u}$ in the $v$-th row in the Klein tableau for $E$.
For each pair $(v,u)$ where $r+2\leq v\leq m$ and $1\leq u\leq r$ one such omission occurs;
together they sum up to yield the extra summand $b$ in the formula in the statement of
the result.
\end{proof}

\subsection{The location of symbols in the category $\mathcal S(n)$}

The following result may help to detect the objects in an 
Auslander-Reiten quiver which have a certain entry in their Klein
tableau.  
Suppose $E\in\mathcal S(n)$ has a symbol $\singlebox{\ell_r}$ in the $m$-th row
of its Klein tableau (so $\ell,r,m,n$ satisfy 
the inequalities $2\leq \ell\leq r+1\leq m\leq n$).  
We have seen in Theorem~\ref{theorem-proto2} that there is a map 
$g:E\to Z$  where $Z=T_2^{m,r}\up^{\ell-2}$ which does not factor through
the sink map $v$ for $Z$ in $\mathcal S_2\up^{\ell-2}$.  
In the next statement we show that there is a corresponding object $C$
depending only on $Z$ and $n$, and a map $f:C\to E$ such that $gf$ does not factor
through $v$.  In this sense, $E$ is ``in between'' $C$ and $Z$.

\begin{theorem}\label{theorem-factor}
Suppose that the integers $\ell, r,m,n$ satisfy
$$2\leq \ell \leq r+1 \leq m \leq n.$$
Let $C\in\mathcal S(n)$ be defined as follows:
The picket or bipicket $Z=T_2^{m,r}\up^{\ell-2}$ 
is either projective (i.e.\ $Z=P_\ell^\ell$) in $\mathcal S_2\up^{\ell-2}$
with sink map $P_{\ell-1}^\ell\stackrel v\to P_\ell^\ell$
in which case we put $C=P_n^n$;
or else $Z$ occurs as end term of an Auslander-Reiten sequence
$0\to X\to Y\stackrel v\to Z\to 0$ in $\mathcal S_2\up^{\ell-2}$, 
in which case we put $C=\tau_{\mathcal S(n)}^{-1} X$.
\begin{enumerate}
\item The $R$-module $$\frac{\Hom_{\mathcal S}(C,Z)}{\Im\Hom_{\mathcal S}(C,v)}$$
      is a one dimensional $k$-vector space.
\item For $E\in \mathcal S(n)$, 
      the map given by composition,
      $$\frac{\Hom(E,Z)}{\Im\Hom(E,v)} \;\times\; \Hom(C,E) \;\longrightarrow
      \; \frac{\Hom(C,Z)}{\Im\Hom(C,v)}$$
      is left non-degenerate.
\end{enumerate}
\end{theorem}

The corresponding result about the entries in the LR-tableau is \cite[Proposition~1]{lr}.

\begin{proof}
For the first statement we verify in each of five cases that the Klein tableau
for $C$ contains exactly one symbol $\singlebox{\ell_r}$ in the $m$-th row.
This implies by Theorem~\ref{theorem-proto2} that the $k$-vector space
$\frac{\Hom_{\mathcal S}(C,Z)}{\Im\Hom_{\mathcal S}(C,v)}$ has dimension one.

\begin{enumerate}
\item[(a)] In the first case where $\ell=r+1=m$, the module 
$$Z=T_2^{m,r}\up^{\ell-2}=P_\ell^\ell$$
is the indecomposable projective object in $\mathcal S_2\up^{\ell-2}$;
the corresponding module  $C=P_n^n$ has the following Klein tableau:

\begin{center}
\setcounter{boxsize}{3}
  \begin{picture}(20,15)
    \put(0,3){$P_n^n:$}
    \put(10,0){\begin{picture}(3,12)\put(0,0){\numbox n} \put(0,6)\vdotbox \put(0,9){\numbox 1}
      \end{picture}}
  \end{picture}
\end{center}
In particular, the $m$-th row has entry $\singlebox{\ell_r}$ where $\ell=m$ and $r=m-1$.

\item[(b)] In this case we assume that $\ell<r+1=m$, so we are dealing with a nonprojective
picket $Z=T_2^{m,r}\up^{\ell-2}= P_\ell^m$.
We have seen in (\ref{section-s2up}) that 
$\tau_{\mathcal S_2\up^{\ell-2}} P_\ell^m=P_{\ell-2}^{m-2}$.
Here and in the remaining cases we use 
\cite[Theorem~5.2, see also Lemma~1.2~(3)]{rs-translation} 
to compute the inverse of the 
Auslander-Reiten translation in $\mathcal S(n)$ as the kernel of the minimal
epimorphism representing the embedding:
\begin{eqnarray*}
 C & = & \tau_{\mathcal S(n)}^{-1} P_{\ell-2}^{m-2} \\
   & = & \ker\mepi (P^{\ell-2}\to P^{m-2}, y\mapsto p^{m-\ell}y)\\
   & = & \ker (P^n\oplus P^{\ell-2} \to P^{m-2},(x,y)\mapsto x+p^{m-\ell}y) \\
   & = & (P^{n+\ell-m}\to P^n\oplus P^{\ell-2}, u\mapsto (p^{m-\ell}u,-u))\\
   & = & ( \; ((p^{m-\ell},-1))\subset P^n\oplus P^{\ell-2}\;)
\end{eqnarray*}
This embedding has the following Klein tableau:

\begin{center}
\setcounter{boxsize}{3}
  \begin{picture}(20,33)
    \put(-10,10){$C:$}
    \put(10,0){\begin{picture}(3,12)\put(0,0){\numbox t} 
        \put(0,6)\vdotbox 
        \put(0,9){\numbox \ell}
        \put(0,12){\numbox{\ell'}}
        \put(3,21){\numbox{\ell''}}
        \put(3,27)\vdotbox
        \put(3,30){\numbox 1}
        \put(-12,23){${}_{m-2}\left\{\makebox(0,9){}\right.$}
        \put(0,15){\rectbox6}
      \end{picture}}
  \end{picture}
\end{center}
Here we abbreviate $\ell'=\ell-1$, $\ell''=\ell-2$ and put $t=n+\ell-m$.
We see that the $m$-th row consists of the symbol $\singlebox{\ell_r}$ where $r=m-1$.

\item[(c)] Here we consider the case where $\ell=r+1<m$. 
Then the translate of the bipicket $Z=T_2^{m,r}\up^{\ell-2}$ is the picket 
$\tau_{\mathcal S_2\up^{\ell-2}}Z = P_{\ell-1}^{m-1}$.
As above, $C=\tau_{\mathcal S(n)}^{-1}P_{\ell-1}^{m-1}$ is computed as 
$$C \;=\; (\; ((p^{m-\ell},-1))\subset P^n\oplus P^{\ell-1}\;);$$
its Klein tableau
\begin{center}
\setcounter{boxsize}{3}
  \begin{picture}(20,33)
    \put(-10,10){$C:$}
    \put(10,0){\begin{picture}(3,12)\put(0,0){\numbox t} 
        \put(0,6)\vdotbox 
        \put(0,9){\numbox \ell}
        \put(3,18){\numbox{\ell'}}
        \put(3,24)\vdotbox
        \put(3,27){\numbox 1}
        \put(-12,20){${}_{m-1}\left\{\makebox(0,9){}\right.$}
        \put(0,12){\rectbox6}
      \end{picture}}
  \end{picture}
\end{center}
has the symbol $\singlebox{\ell_r}$ where $r=\ell-1$ in the $m$-th row ($t=n+\ell-m$).

\item[(d$_1$)] We assume $\ell+1=r+1<m$. 
The translate of the bipicket $Z=T_2^{m,r}\up^{\ell-2}$ is the bipicket
$\tau_{\mathcal S_2\up^{\ell-2}} Z = T_2^{m-1,r-1}\up^{\ell-2}$.
Note that here --- unlike in the next case  --- 
the factor of the embbedding defining this module is cyclic.  We compute
\begin{eqnarray*}
\qquad C & = & \tau_{\mathcal S(n)}^{-1} T_2^{m-1,r-1}\up^{\ell-2} \\
  & = & \ker\mepi ( P^\ell \oplus P^{\ell-2} \to P^{m-1}\oplus P^{r-1},\\
  &   &  \qquad      (y,z)\mapsto (p^{m-1-\ell}y,p^{r-\ell}y+p^{r-\ell+1}z)) \\
  & = & \ker\mepi ( P^\ell \oplus P^{\ell-2} \to P^{m-1}\oplus P^{r-1},\\
  &   &  \qquad      (y,z)\mapsto (p^{m-1-\ell}y,y+pz)) \\
  & = & \ker ( P^n\oplus P^\ell \oplus P^{\ell-2}\to P^{m-1}\oplus P^{r-1}, \\
  &   & \qquad       (x,y,z)\mapsto (x+p^{m-1-\ell}y, y+pz)) \\
  & = & (P^{n+\ell-m} \to P^n\oplus P^\ell \oplus P^{\ell-2},
                u \mapsto (p^{m-\ell}u,-pu, u) ) \\
  & = & (\;((p^{m-\ell},-p,1))\subset P^n\oplus P^\ell \oplus P^{\ell-2}\;)
\end{eqnarray*}
The Klein tableau is as follows:
\begin{center}
\setcounter{boxsize}{3}
  \begin{picture}(20,36)
    \put(-10,10){$C:$}
    \put(10,0){\begin{picture}(3,12)\put(0,0){\numbox t} 
        \put(0,6)\vdotbox 
        \put(0,9){\numbox \ell}
        \put(3,18){\numbox{\ell'}}
        \put(6,24){\numbox{\ell''}}
        \put(6,30)\vdotbox
        \put(6,33){\numbox 1}
        \put(-12,23){${}_{m-1}\left\{\makebox(0,13){}\right.$}
        \put(0,12){\rectboxL8}
        \put(3,18){\rectboxJ6}
        \put(10,26){$\left.\makebox(0,10){}\right\}{\!}_\ell$}
      \end{picture}}
  \end{picture}
\end{center}
Namely, the sequence of radical layers of the elements $p^iu$ in the total space
is $1, 2, \ldots, \ell-2,\ell,m,m+1,\ldots,t=n+\ell-m$.
So the $m$-th row consists of the symbol $\singlebox{\ell_r}$ where $r=\ell$.

\item[(d$_2$)] Finally we consider the case where $\ell+1<r+1<m$.
The translate of $Z=T_2^{m,r}\up^{\ell-2}$ is the bipicket 
$\tau_{\mathcal S_2\up^{\ell-2}} Z = T_2^{m-1,r-1}\up^{\ell-2}$.  We compute
\begin{eqnarray*}
\qquad \qquad C & = & \tau_{\mathcal S(n)}^{-1} T_2^{m-1,r-1}\up^{\ell-2} \\
  & = & \ker\mepi (P^\ell\oplus P^{\ell-2} \to P^{m-1} \oplus P^{r-1},\\
  &   & \qquad  (y,z)\mapsto (p^{m-1-\ell}y,p^{r-\ell}y+p^{r-\ell+1}z)) \\
  & = & \ker (P^n\oplus P^n \oplus P^\ell \oplus P^{\ell-2}\to P^{m-1}\oplus P^{r-1},\\
  &   & \qquad  (w,x,y,z) \mapsto \\
  &   & \qquad (w+p^{m-1-r}x+p^{m-1-\ell}y,x+p^{r-\ell}y+p^{r-\ell+1}z))\\
  & = & (P^{n-m+\ell}\oplus P^{n-r+\ell}\to P^n\oplus P^n\oplus P^\ell\oplus P^{\ell-2},\\
  &   &  \qquad (u,v)\mapsto (p^{m-\ell}u,p^{r-\ell}v, -pu-v,u))
\end{eqnarray*}
The Klein tableau is as follows:
\begin{center}
\setcounter{boxsize}{3}
  \begin{picture}(20,45)
    \put(-10,10){$C:$}
    \put(10,0){\begin{picture}(3,12)
        \put(0,0){\numbox t} 
        \put(3,0){\numbox {t'}}
        \put(0,6)\vdotbox 
        \put(3,3){\rectbox5}
        \put(3,12)\vdotboxI
        \put(0,9){\numbox{\ell^*}}
        \put(0,12){\numbox \ell}
        \put(0,15){\rectboxL{10}}
        \put(3,18){\numbox{\ell^*}}
        \put(3,21){\numbox{\ell'}}
        \put(3,24){\rectboxJ7}
        \put(6,27){\numbox{\ell}}
        \put(6,30){\numbox{\ell'}}
        \put(6,33){\numbox{\ell''}}
        \put(9,33){\numbox{\ell''}}
        \put(6,39)\vdotbox
        \put(9,39)\vdotbox
        \put(6,42){\numbox 1}
        \put(9,42){\numbox 1}
        \put(-12,29){${}_{m-1}\left\{\makebox(0,15){}\right.$}
        \put(13,34){$\left.\makebox(0,12){}\right\}{\!}_{r-1}$}
      \end{picture}}
  \end{picture}
\end{center}
Here $\ell^*=\ell+1$, $t=n-m+\ell$ and $t'=n-r+\ell$.
We determine the two symbols corresponding to the labels $\ell$. 
Write $C$ as the embedding $(A\subset B)$ and verify that the types 
of $B/p^{\ell-2}A$, $B/p^{\ell-1}A$ and $B/p^\ell A$ are 
$(m-1,r-1,\ell-2,\ell-2$); $(m-1, r,\ell-1,\ell-2)$; and $(m,r,\ell,\ell-2)$,
respectively.  
Thus, the two labels $\ell$ occur in rows $m$ and $\ell$, while the two
labels $\ell-1$ occur in rows $r$ and $\ell-1$. Since $r>\ell$, 
the subscript $r$ remains for the label $\ell$ in row $m$.  
So the two symbols  $\singlebox{\ell_r}$
and $\singlebox{\ell_{\ell'}}$ occur in rows $m$ and $\ell$, respectively.
\end{enumerate}

For the proof of the second statement, we first consider the case where $Z=P_\ell^\ell$
is projective in $\mathcal S_2\up^{\ell-2}$. 
Let $\pi:P^n_n\to P_\ell^\ell$ be the canonical map.
Suppose a map $f:E\to P_\ell^\ell$ with $E\in\mathcal S(n)$ 
does not factor through the sink map
$P_{\ell-1}^\ell\to P_\ell^\ell$; then $f$ is an epimorphism.
Recall that $P_n^n$ is a projective object in the abelian category $\mathcal H(n)$
of all maps between $R/(p^n)$-modules; and that $\mathcal S(n)\subset \mathcal H(n)$
is a full subcategory.  It follows that $\pi$ factors through the epimorphism $f$. 

\smallskip
For the case where $Z$ is nonprojective, we adapt the second part of the proof of 
\cite[Proposition~2]{lr}:
Suppose that $E\in\mathcal S(n)$ is given.
Let $\mathcal A:0\to X\stackrel u\to Y\stackrel v\to Z\to 0$ be 
the Auslander-Reiten sequence in $\mathcal S_2\up^{\ell-2}$ ending at $Z$,
as given in Proposition~\ref{proposition-sink}. Let 
$$ \mathcal E:\quad 0\to X\stackrel f\to B\stackrel g\to C\to 0$$
be the Auslander-Reiten sequence in $\mathcal S(n)$ starting at $X$,
its end term is $C$. 
Since $\mathcal A$ is nonsplit, there are maps 
$h':B\to Y$, $h:C\to Z$ which make the upper part
of the following diagram commutative.
$$
\begin{CD}
\mathcal E:\quad @.  0 @>>> X @>f>> B @>g>> C @>>> 0 \\
@. @.  @|   @V{h'}VV  @VV{h}V   \\
\mathcal A:\quad @.  0 @>>> X @>u>> Y @>v>> Z @>>> 0\\
@.  @.  @|   @A{p'}AA @AA{p}A \\
p\mathcal A:\quad @. 0 @>>> X @>{s}>> L @>{t}>> E @>>> 0 \\
\end{CD}
$$
In order to show that the bilinear form given by composition 
$$\frac{\Hom(E,Z)}{\Im\Hom(E,v)}\times \Hom(C,E)\longrightarrow
  \frac{\Hom(C,Z)}{\Im\Hom(C,v)}$$
is left non-degenerate,
let $p:E\to Z$ be a map which does not factor through $v$. 
We will construct $q:C\to E$ such that $pq$ does not factor through $v$.
Since $p$ does not factor through $v$, the induced sequence at the bottom 
of the above diagram does not split. Hence the map $s$ factors through $f$:
There is a map $q':B\to L$ such that $s=q'f$.  Let $q:C\to E$ be the cokernel map,
so $qg=tq'$.  Then $pqg=ptq'=vp'q'$.  Since $p'q'f=p's=u=h'f$,
there exists $z:C\to Y$ such that $zg=p'q'-h'$.
So $pqg=vp'q'=v(zg+h')=(vz+h)g$ and since $g$ is onto,
$pq=vz +h$.  Since $\mathcal E$ is not split exact, $h$ does not factor through
$v$, and hence  $pq$ does not factor through $v$.
\end{proof}

\subsection{Examples}\label{subsection-S5}

\def\GammaFive{\multiput{\ssq} at 1.9 -.1  1.9 .1 /  
\multiput{\ssq} at 3.9 -.4  3.9 -.2  3.9 0  3.9 .2  3.9 .4 /
\put{\num1} at 4 0
\put{\num2} at 4 -.2
\put{\num3} at 4 -.4
\multiput{\ssq} at 5.9 -.2  5.9 0  5.9 .2 /
\put{\num1} at 6 .2
\put{\num2} at 6 0
\put{\num3} at 6 -.2
\multiput{\ssq} at 7.9 -.2  7.9 0  7.9 .2 /
\multiput{\ssq} at 9.9 -.4  9.9 -.2  9.9 0  9.9 .2  9.9 .4 /
\put{\num1} at 10 -.2
\put{\num2} at 10 -.4
\multiput{\ssq} at 11.9 -.1  11.9 .1 /
\put{\num1} at 12 .1
\put{\num2} at 12 -.1 
\multiput{\ssq} at .8 1.3  .8 1.5  .8 1.7  1 1.7 /
\put{\num1} at 1.1 1.7
\put{\num2} at .9 1.3
\multiput{\ssq} at  2.8 1.1  2.8 1.3  2.8 1.5  2.8 1.7  2.8 1.9  3 1.7  3 1.9 /
\put{\num1} at 3.1 1.7
\put{\num2} at 2.9 1.3
\put{\num3} at 2.9 1.1
\multiput{\ssq} at 4.8 1.1  4.8 1.3  4.8 1.5  4.8 1.7  4.8 1.9  
                5 1.9  5 1.5  5 1.7 /
\multiput{\num1} at 4.9 1.5  5.1 1.9 /
\multiput{\num2} at 4.9 1.3  5.1 1.7 /
\put{\num4} at 4.9 1.1
\put{\num3} at 5.1 1.5
\multiput{\ssq} at 6.8 1.2  6.8 1.4  6.8 1.6  6.8 1.8  7 1.6  7 1.8 /
\put{\num1} at 7.1 1.8
\put{\num2} at 7.1 1.6
\put{\num3} at 6.9 1.2
\multiput{\ssq} at  9 1.5  9 1.7  9 1.9  
                8.8 1.1  8.8 1.3  8.8 1.5  8.8 1.7  8.8 1.9 /
\put{\num1} at 9.1 1.5
\put{\num2} at 8.9 1.1 
\multiput{\ssq} at 10.8 1.1  10.8 1.3  10.8 1.5  10.8 1.7  10.8 1.9  
                                11 1.7  11 1.9 /
\multiput{\num1} at 11.1 1.9  10.9 1.3 /
\put{\num2} at 11.1 1.7 
\put{\num3} at 10.9 1.1 
\multiput{\ssq} at 12.8 1.3  12.8 1.5  12.8 1.7  13 1.7 /
\put{\num1} at 13.1 1.7
\put{\num2} at 12.9 1.3
%
%
\multiput{\ssq} at .8 2.2  .8 2.4  .8 2.6  .8 2.8  .8 3  1 3 /
\put{\num1} at 1.1 3
\put{\num2} at .9 2.4
\put{\num3} at .9 2.2
\multiput{\ssq} at 1.7 2.6  1.7 2.8  1.7 3  1.7 3.2  1.7 3.4  
                1.9 3.4 1.9 3  1.9 3.2  2.1 3.4 /
\multiput{\num1} at 2 3.2  2.2 3.4 /
\multiput{\num2} at 2 3  1.8 2.8 /
\put{\num3} at 1.8 2.6
\multiput{\ssq} at 2.9 2.4  2.9 2.6  2.9 2.8 /
\put{\num1} at 3 2.6
\put{\num2} at 3 2.4
\multiput{\ssq} at 3.7 2.6  3.7 2.8  3.7 3  3.7 3.2  3.7 3.4  
                3.9 3.4  3.9 3  3.9 3.2   4.1 3.4  4.1 3.2 /
\multiput{\num1} at 4.2 3.4  4 3.2 /
\multiput{\num2} at 4.2 3.2  3.8 2.8 /
\put{\num3} at 4 3
\put{\num4} at 3.8 2.6
\multiput{\ssq} at 4.8 2.2  4.8 2.4  4.8 2.6  4.8 2.8  4.8 3  5 2.8  5 3 /
\put{\num1} at 5.1 3
\put{\num2} at 5.1 2.8
\put{\num3} at 4.9 2.4
\put{\num4} at 4.9 2.2
\multiput{\ssq} at 5.7 2.6  5.7 2.8  5.7 3  5.7 3.2  5.7 3.4  
                5.9 2.8  5.9 3  5.9 3.2  5.9 3.4
                6.1 3.4  6.1 3.2 /
\multiput{\num1} at 6.2 3.4  6.0 3 /
\multiput{\num2} at 6.2 3.2  5.8 2.8 /
\put{\num3} at 6.0 2.8 
\put{\num4} at 5.8 2.6
\multiput{\ssq} at 6.9 2.3  6.9 2.5  6.9 2.7  6.9 2.9 /
\put{\num1} at 7 2.5
\put{\num2} at 7 2.3
\multiput{\ssq} at 7.7 2.6  7.7 2.8  7.7 3  7.7 3.2  7.7 3.4  
                7.9 2.8  7.9 3  7.9 3.2  7.9 3.4   8.1 3.2  8.1 3.4 /
\multiput{\num1} at 8.2 3.4  8 3 /
\multiput{\num2} at 8.2 3.2  7.8 2.6 /
\put{\num3} at 8 2.8
\multiput{\ssq} at 8.8 2.2  8.8 2.4  8.8 2.6  8.8 2.8  8.8 3  9 2.8  9 3  /
\put{\num1} at 9.1 3
\put{\num2} at 9.1 2.8
\put{\num3} at 8.9 2.2
\multiput{\ssq} at 9.7 2.6  9.7 2.8  9.7 3  9.7 3.2  9.7 3.4 
                9.9 3  9.9 3.2  9.9 3.4  10.1 3.2  10.1 3.4 /
\multiput{\num1} at 10.2 3.4  10 3 /
\put{\num2} at 10.2 3.2
\put{\num3} at 9.8 2.6
\multiput{\ssq} at 10.9 2.4  10.9 2.6  10.9 2.8  /
\put{\num1} at 11 2.4
\multiput{\ssq} at 11.7 2.6  11.7 2.8  11.7 3  11.7 3.2  11.7 3.4  
                11.9 3.4  11.9 3  11.9 3.2       12.1 3.4 /
\multiput{\num1} at 12.2 3.4  11.8 2.8 /
\put{\num2} at 12 3
\put{\num3} at 11.8 2.6
\multiput{\ssq} at 12.8 2.2  12.8 2.4  12.8 2.6  12.8 2.8  12.8 3  13 3 /
\put{\num1} at 13.1 3
\put{\num2} at 12.9 2.4
\put{\num3} at 12.9 2.2
\multiput{\ssq} at .8 3.6  .8 3.8  .8 4  .8 4.2  .8 4.4  1 4  1 4.4  1 4.2 /
\multiput{\num1} at .9 3.8  1.1 4.2 /
\put{\num2} at 1.1 4 
\put{\num3} at .9 3.6
\multiput{\ssq} at  2.7 3.6  2.7 3.8  2.7 4  2.7 4.2  2.7 4.4  
                2.9 4.4  2.9 4  2.9 4.2      3.1 4.4 /
\multiput{\num1} at 3 4.4  3.2 4.4 /
\multiput{\num2} at 3 4.2  2.8 3.8 /
\put{\num3} at 3 4 
\put{\num4} at 2.8 3.6
\multiput{\ssq} at 4.8 3.7  4.8 3.9  4.8 4.1  4.8 4.3  5 4.3  5 4.1 /
\put{\num1} at 5.1 4.1
\put{\num2} at 4.9 3.7
\multiput{\ssq} at  6.7 3.6  6.7 3.8  6.7 4  6.7 4.2  6.7 4.4  
                6.9 3.6  6.9 3.8  6.9 4  6.9 4.2  6.9 4.4   7.1 4.2  7.1 4.4 /
\multiput{\num1} at 7.2 4.4  7 4 /
\multiput{\num2} at 7.2 4.2  6.8 3.6 /
\put{\num3} at 7 3.8
\put{\num4} at 7 3.6
\multiput{\ssq} at   9 4.3  9 4.1  8.8 3.7  8.8 3.9  8.8 4.1  8.8 4.3 /
\multiput{\num1} at 9.1 4.3  8.9 3.9 /
\put{\num2} at 9.1 4.1
\put{\num3} at 8.9 3.7
\multiput{\ssq} at    10.7 3.6  10.7 3.8  10.7 4  10.7 4.2  10.7 4.4  
                10.9 4.4   10.9 4  10.9 4.2   11.1 4.4 /
\put{\num1} at 11.2 4.4 
\put{\num2} at 11 4
\put{\num3} at 10.8 3.6
\multiput{\ssq} at 12.8 3.6  12.8 3.8  12.8 4  12.8 4.2  12.8 4.4  13 4  13 4.4  13 4.2 /
\multiput{\num1} at 12.9 3.8  13.1 4.2 /
\put{\num2} at 13.1 4 
\put{\num3} at 12.9 3.6
\multiput{\ssq} at 1.8 4.6  1.8 4.8  1.8 5  1.8 5.2  1.8 5.4  
                2 5.4  2 5  2 5.2 /
\multiput{\num1} at 2.1 5.4  1.9 4.8 /
\put{\num2} at 2.1 5.2
\put{\num3} at 2.1 5
\put{\num4} at 1.9 4.6
\multiput{\ssq} at 3.8 4.7  3.8 4.9  3.8 5.1  3.8 5.3  4 5.3 /
\put{\num1} at 4.1 5.3
\put{\num2} at 3.9 4.7
\multiput{\ssq} at 5.8 4.6  5.8 4.8  5.8 5  5.8 5.2  5.8 5.4  6 5.2  6 5.4 /
\put{\num1} at 6.1 5.2 
\put{\num2} at 5.9 4.6
\multiput{\ssq} at  7.8 4.6  7.8 4.8  7.8 5  7.8 5.2  7.8 5.4  8 5.2  8 5.4 /
\multiput{\num1} at 8.1 5.4  7.9 5 /
\put{\num2} at 8.1 5.2
\put{\num3} at 7.9 4.8
\put{\num4} at 7.9 4.6
\multiput{\ssq} at  9.8 4.7  9.8 4.9  9.8 5.1  9.8 5.3  10 5.3 /
\put{\num1} at 10.1 5.3
\put{\num2} at 9.9 4.9
\put{\num3} at 9.9 4.7
\multiput{\ssq} at 11.8 4.6  11.8 4.8  11.8 5  11.8 5.2  11.8 5.4  
                12 5.4  12 5  12 5.2 /
\put{\num1} at 12.1 5.2
\put{\num2} at 12.1 5
\put{\num3} at 11.9 4.6
\multiput{\ssq} at .8 5.6  .8 5.8  .8 6  .8 6.2  .8 6.4  1 6.4  1 6  1 6.2 /
\put{\num1} at 1.1 6.4
\put{\num2} at 1.1 6.2
\put{\num3} at 1.1 6
\put{\num4} at .9 5.6
\multiput{\ssq} at  2.9 5.7  2.9 5.9  2.9 6.1  2.9 6.3 /
\put{\num1} at 3 5.7
\multiput{\ssq} at  4.8 5.6  4.8 5.8  4.8 6  4.8 6.2  4.8 6.4  5 6.4 /
\put{\num1} at 5.1 6.4
\put{\num2} at 4.9 5.6
\multiput{\ssq} at  6.9 5.9  6.9 6.1 /
\put{\num1} at 7 5.9
\multiput{\ssq} at  8.8 5.6  8.8 5.8  8.8 6  8.8 6.2  8.8 6.4  9 6.4 /
\put{\num1} at 9.1 6.4
\put{\num2} at 8.9 6
\put{\num3} at 8.9 5.8
\put{\num4} at 8.9 5.6
\multiput{\ssq} at 10.9 5.7  10.9 5.9  10.9 6.1  10.9 6.3 /
\put{\num1} at 11 6.1
\put{\num2} at 11 5.9
\put{\num3} at 11 5.7
\multiput{\ssq} at 12.8 5.6  12.8 5.8  12.8 6  12.8 6.2  12.8 6.4  13 6.4  13 6  13 6.2 /
\put{\num1} at 13.1 6.4
\put{\num2} at 13.1 6.2
\put{\num3} at 13.1 6
\put{\num4} at 12.9 5.6
\multiput{\ssq} at   1.9 6.7  1.9 6.9  1.9 7.1  1.9 7.3 /
\multiput{\ssq} at    2.9 7.6  2.9 7.8  2.9 8  2.9 8.2  2.9 8.4 /
\multiput{\ssq} at  3.9 6.6  3.9 6.8  3.9 7  3.9 7.2  3.9 7.4 /
\put{\num1} at 4 6.6
\multiput{\ssq} at  5.9 7 /
\put{\num1} at 6 7
\multiput{\ssq} at   7.9 7 /
\multiput{\ssq} at  9.9 6.6  9.9 6.8  9.9 7  9.9 7.2  9.9 7.4 /
\put{\num1} at 10 7.2 
\put{\num2} at 10 7
\put{\num3} at 10 6.8
\put{\num4} at 10 6.6
\multiput{\ssq} at  10.9 7.6  10.9 7.8  10.9 8  10.9 8.2  10.9 8.4 /
\put{\num1} at 11 8.4
\put{\num2} at 11 8.2
\put{\num3} at 11 8
\put{\num4} at 11 7.8
\put{\num5} at 11 7.6
\multiput{\ssq} at  11.9 6.7  11.9 6.9  11.9 7.1  11.9 7.3 /
\put{\num1} at 12 7.3
\put{\num2} at 12 7.1
\put{\num3} at 12 6.9
\put{\num4} at 12 6.7
\arr{1.3 1.05} {1.7 0.45} 
\arr{2.3 0.45} {2.7 1.05} 
\arr{3.3 1.05} {3.7 0.45} 
\arr{4.3 0.45} {4.7 1.05} 
\arr{5.3 1.05} {5.7 0.45} 
\arr{6.3 0.45} {6.7 1.05} 
\arr{7.3 1.05} {7.7 0.45} 
\arr{8.3 0.45} {8.7 1.05} 
\arr{9.3 1.05} {9.7 0.45} 
\arr{10.3 0.45} {10.7 1.05} 
\arr{11.3 1.05} {11.7 0.45} 
\arr{12.3 0.45} {12.7 1.05} 
\arr{1.3 1.95} {1.6 2.4} 
\arr{2.3 2.55} {2.7 1.95} 
\arr{3.3 1.95} {3.6 2.4} 
\arr{4.3 2.55} {4.7 1.95} 
\arr{5.3 1.95} {5.7 2.55} 
\arr{6.4 2.4 } {6.7 1.95} 
\arr{7.3 1.95} {7.6 2.4 } 
\arr{8.4 2.4 } {8.7 1.95} 
\arr{9.3 1.95} {9.7 2.55} 
\arr{10.3 2.55} {10.7 1.95} 
\arr{11.3 1.95} {11.7 2.55} 
\arr{12.3 2.55} {12.7 1.95} 
\arr{1.3 3.7} {1.6 3.4} 
\arr{2.3 3.3} {2.6 3.6} 
\arr{3.3 3.7} {3.6 3.4} 
\arr{4.4 3.4} {4.7 3.7} 
\arr{5.3 3.7} {5.6 3.4} 
\arr{6.4 3.4} {6.6 3.6} 
\arr{7.4 3.6} {7.6 3.4} 
\arr{8.4 3.4} {8.7 3.7} 
\arr{9.3 3.7} {9.7 3.3} 
\arr{10.4 3.4} {10.7 3.7} 
\arr{11.4 3.6} {11.7 3.3} 
\arr{12.4 3.4} {12.7 3.7} 
\arr{1.3 4.3} {1.7 4.7} 
\arr{2.3 4.7} {2.6 4.4} 
\arr{3.3 4.3} {3.7 4.7} 
\arr{4.3 4.7} {4.7 4.3} 
\arr{5.3 4.3} {5.7 4.7} 
\arr{6.3 4.7} {6.6 4.4} 
\arr{7.4 4.4} {7.7 4.7} 
\arr{8.3 4.7} {8.7 4.3} 
\arr{9.3 4.3} {9.7 4.7} 
\arr{10.3 4.7} {10.7 4.3} 
\arr{11.4 4.4} {11.7 4.7} 
\arr{12.3 4.7} {12.7 4.3} 
\arr{1.3 5.7} {1.7 5.3} 
\arr{2.3 5.3} {2.7 5.7} 
\arr{3.3 5.7} {3.7 5.3} 
\arr{4.3 5.3} {4.7 5.7} 
\arr{5.3 5.7} {5.7 5.3} 
\arr{6.3 5.3} {6.7 5.7} 
\arr{7.3 5.7} {7.7 5.3} 
\arr{8.3 5.3} {8.7 5.7} 
\arr{9.3 5.7} {9.7 5.3} 
\arr{10.3 5.3} {10.7 5.7} 
\arr{11.3 5.7} {11.7 5.3} 
\arr{12.3 5.3} {12.7 5.7} 
\arr{1.3 6.3} {1.7 6.7} 
\arr{2.3 6.7} {2.7 6.3} 
\arr{3.3 6.3} {3.7 6.7} 
\arr{4.3 6.7} {4.7 6.3} 
\arr{5.3 6.3} {5.7 6.7} 
\arr{6.3 6.7} {6.7 6.3} 
\arr{7.3 6.3} {7.7 6.7} 
\arr{8.3 6.7} {8.7 6.3} 
\arr{9.3 6.3} {9.7 6.7} 
\arr{10.3 6.7} {10.7 6.3} 
\arr{11.3 6.3} {11.7 6.7} 
\arr{12.3 6.7} {12.7 6.3} 

\arr{2.3 7.3} {2.7 7.7} 
\arr{3.3 7.7} {3.7 7.3} 
\arr{10.3 7.3} {10.7 7.7} 
\arr{11.3 7.7} {11.7 7.3} 

\arr{1.3 2.72}{1.6 2.84}
\arr{2.4 2.84}{2.7 2.72}
\arr{3.3 2.72}{3.6 2.84}
\arr{4.4 2.84}{4.7 2.72}
\arr{5.3 2.72}{5.6 2.84}
\arr{6.4 2.84}{6.7 2.72}
\arr{7.3 2.72}{7.6 2.84}
\arr{8.4 2.84}{8.7 2.72}
\arr{9.3 2.72}{9.6 2.84}
\arr{10.4 2.84}{10.7 2.72}
\arr{11.3 2.72}{11.6 2.84}
\arr{12.4 2.84}{12.7 2.72}
\setdots<2pt>
\plot 1 0  1.7 0 /
\plot 2.3 0  3.7 0 /
\plot 4.3 0  5.7 0 /
\plot 6.3 0  7.7 0 /
\plot 8.3 0  9.7 0 /
\plot 10.3 0  11.7 0 /
\plot 12.3 0  13 0 /
\plot 1.3 2.6  1.6 2.6 /
\plot 2 2.6  2.7 2.6 /
\plot 3.3 2.6  3.6 2.6 /
\plot 4 2.6  4.7 2.6 /
\plot 5.3 2.6  5.6 2.6 /
\plot 6 2.6  6.7 2.6 /
\plot 7.3 2.6  7.6 2.6 /
\plot 8 2.6  8.7 2.6 /
\plot 9.3 2.6  9.6 2.6 /
\plot 10 2.6  10.7 2.6 /
\plot 11.3 2.6  11.6 2.6 /
\plot 12 2.6  12.7 2.6 /
\plot 1 7  1.7 7 /
\plot 4.3 7  5.7 7 /
\plot 6.3 7  7.7 7 /
\plot 8.3 7  9.7 7 /
\plot 12.3 7  13 7 /
\setsolid
\plot 1 0  1 1.1 /
\plot 1 1.9  1 2 /
\plot 1 3.2  1 3.4 /
\plot 1 4.6  1 5.4 /
\plot 1 6.6  1 7 /
\plot 13 0  13 1.1 /
\plot 13 1.9  13 2 /
\plot 13 3.2  13 3.4 /
\plot 13 4.6  13 5.4 /
\plot 13 6.6  13 7 / }

(1) The diagram below represents the Auslander-\-Rei\-ten quiver for the category
$\mathcal S(5)$, which has the largest number of indecomposable objects
among all representation finite categories of type $\mathcal S(n)$.
We refer to \cite[Section~6.5]{rs-inv} for a detailed description of the 
category $\mathcal S(5)$ and its objects.
In the diagram, each object is represented by its LR-tableau.  Note that each LR-tableau
of an indecomposable object in $\mathcal S(5)$ 
can be refined uniquely to a Klein tableau, so we may omit the 
subscripts of the labels.

\begin{figure}[hb]
$$
\hbox{\beginpicture
\setcoordinatesystem units <.95cm,1.05cm>
\GammaFive
\setsolid
\setquadratic
\plot 1.5 1.6  .5 2.6  1.5 3.6  2.3 4.2  3 5  4 5.7  5 5  5.7 4.2  6.5 3.6  7.3 2.6  
      6.5 1.6  5.5 1  5 .4  4 -.7  3 .4  2.5 1  1.5 1.6 /
\ellipticalarc axes ratio 3:5 360 degrees from 4.9 3.35 center at 4.9 2.7
\ellipticalarc axes ratio 3:5 360 degrees from 3.1 3.35 center at 3.1 2.7
\put{$\mathcal R$} at .4 3
\put{$\scriptstyle Y_2$} at 3.1 3.1 
\put{$\scriptstyle Y_1$} at 3.0 6.6
\put{$\scriptstyle Z$} at 6.75 2.95
\put{$\scriptstyle X$} at 11 3.1
\put{$\scriptstyle C$} at 1.2 2.4

\endpicture}
$$
\end{figure}

\smallskip
Let us consider as in \cite{lr} 
the encircled region $\mathcal R$ consisting of all indecomposables 
which have an entry $\singlebox 2$ in their 4th row. Note that the two ``eyes'' are
not part of the region $\mathcal R$.  Let
$$0\to X\to Y_1\oplus Y_2 \stackrel v\to Z\to 0$$
be the Auslander-Reiten sequence in the category $\mathcal S_1\up^1$ ending at
$Z=P_2^4$; here $X=P_1^3$, $Y_1=P_1^4$ and $Y_2=P_2^3$.   The modules $E$
in $\mathcal R$ are characterized by the existence of a homomorphism
in $\Hom(E,Z)$ which does not factor over $v$ \cite[Theorem 1]{lr},
and also by the existence of a map 
$C=\tau_{\mathcal S(5)}^{-1}X\to Z$, not in $\Im\Hom(C,v)$, which  
factors through $E$ \cite[Proposition 2]{lr}.

\smallskip
(2) As Klein tableaux are refinements of LR-tableaux, the entries in the
Klein tableau provide an even finer selection of indecomposable objects.
When adding subscripts to the entries in the LR-tableau, the entry
$\singlebox 2$ in the 4th row may become one of the symbols
$\singlebox{2_1}$, $\singlebox{2_2}$, or $\singlebox{2_3}$.
In fact, the three modules along the middle diagonal, which are encircled
in the second diagram, all have a box $\singlebox{2_2}$ in the 4th row
of their Klein tableau.   The modules above them will carry a $\singlebox{2_1}$,
and those below them a $\singlebox{2_3}$.

\begin{figure}[ht]
$$
\hbox{\beginpicture
\setcoordinatesystem units <.95cm,1.05cm>
\GammaFive
\setdots<1mm>
\setquadratic
\plot 1.5 1.6  .5 2.6  1.5 3.6  2.3 4.2  3 5  4 5.7  5 5  5.7 4.2  6.5 3.6  7.3 2.6  
      6.5 1.6  5.5 1  5 .4  4 -.7  3 .4  2.5 1  1.5 1.6 /
\setsolid
\plot 2.4 1.5  2.6 1.9  3 2.2  3.3 2.5  3.5 2.9  3.75 3.5  
      4 3.9  4.25 4.2  4.5 4.4
      4.75 4.6  5 4.65  5.25 4.6  5.5 4.5  5.6 4.1  5.5 3.8  5.25 3.55  5 3.4
      4.5 2.9  4 2.3  3.75 1.9  3.5 1.5  3.3 1.1  3 .9  2.55 1  2.4 1.5 /
\setdots<1mm>
\ellipticalarc axes ratio 3:5 360 degrees from 4.9 3.35 center at 4.9 2.7
\ellipticalarc axes ratio 3:5 360 degrees from 3.1 3.35 center at 3.1 2.7
\put{$\scriptstyle Y_2$} at 3.1 3.1 
\put{$\scriptstyle Y_1$} at 3.6 5
\put{$\scriptstyle Y_3$} at 2 .4
\put{$\scriptstyle Z$} at 4.6 4
\put{$\scriptstyle X$} at .6 1.5
\put{$\scriptstyle C$} at 2.6 1.5

\endpicture}
$$
\end{figure}

\smallskip
Returning to the modules which have a $\singlebox{2_2}$ in the 4th row
of their Klein tableau, we consider the Auslander-Reiten sequence
in $\mathcal S_2$ ending at the bipicket $Z=T_2^{4,2}$:
$$0\to X\to Y_1\oplus Y_2\oplus Y_3\stackrel v\to Z\to 0$$
Here, $X=T_2^{3,1}$, $Y_1=T_2^{4,1}$, $Y_2=P_2^3$ and $Y_3=P_0^2$.
According to Theorem~\ref{theorem-prototypes}, each of the
modules $E$ in the marked region 
will admit a map $E\to Z$ which does not factor through $v$.
And according to Theorem~\ref{theorem-factor}, there is a map
$C=\tau_{\mathcal S(5)}^{-1}X \to Z$, not in $\Im\Hom(C,v)$, which 
factors through $E$. 

\smallskip
(3) In order to determine the modules which carry a symbol $\singlebox{\ell_r}$
where  $\ell>2$, the Auslander-Reiten sequences from (\ref{section-s2up})
can be used.  We consider as example the modules which have a 
$\singlebox{4_3}$ in row 5.  They are encircled in the
third diagram.

\begin{figure}[ht]
$$
\hbox{\beginpicture
\setcoordinatesystem units <.95cm,1.05cm>
\GammaFive
\setquadratic
\setsolid
\plot .6 6  .65 5.5  .8 5.2  1.5 4.5  2.6 3.4  3.6 2.4  4.7 .9  
      5 .7  5.3 .9  5.4 1.5  5.3 2  5.27 2.07  5.2 2.1  5 2.1  4.8 2.1  
      4.66 2.17  4.6 2.3  4.5 2.9  4.4 3.5  4 4  3.5 4.5  
      2.5 5.5  1.5 6.5  1.2 6.7  1 6.7  .7 6.4  .6 6 /
\put{$\scriptstyle Y_2$} at 2.7 3.1 
\put{$\scriptstyle Y_1$} at 12 7.7
\put{$\scriptstyle Y_3$} at 4 .7
\put{$\scriptstyle Z$} at 4.55 1.5
\put{$\scriptstyle X$} at 11 6.7
\put{$\scriptstyle C$} at 1.5 6
\put{$\scriptstyle C$} at 13.5 6

\endpicture}
$$
\end{figure}

\smallskip
Consider the module $Z=T_2^{5,3}\up^2$ on the right hand side of the 
marked region.  It occurs as the end term of the Auslander-Reiten sequence
in $\mathcal S_2\up^2$:
$$0\to X \to Y_1\oplus Y_2 \oplus Y_3 \stackrel v\to Z\to 0,$$
here $X=P_3^4$, $Y_1=P_4^4$, $Y_2=P_2^3$, and $Y_3=P_3^5$; this
sequence is pictured on the left hand side in the partial Auslander-Reiten
quiver in (\ref{section-s2up}).  
As predicted by Theorem~\ref{theorem-prototypes}, the modules in the 
region are those which admit a map into $Z$ which does not factor over $v$.
Putting $C=\tau_{\mathcal S(5)}^{-1}(X)$, the modules in this region are
those $E$ for which there is a map $C\to Z$, not in $\Im\Hom(C,v)$, which
factors through $E$ (see Theorem~\ref{theorem-factor}).


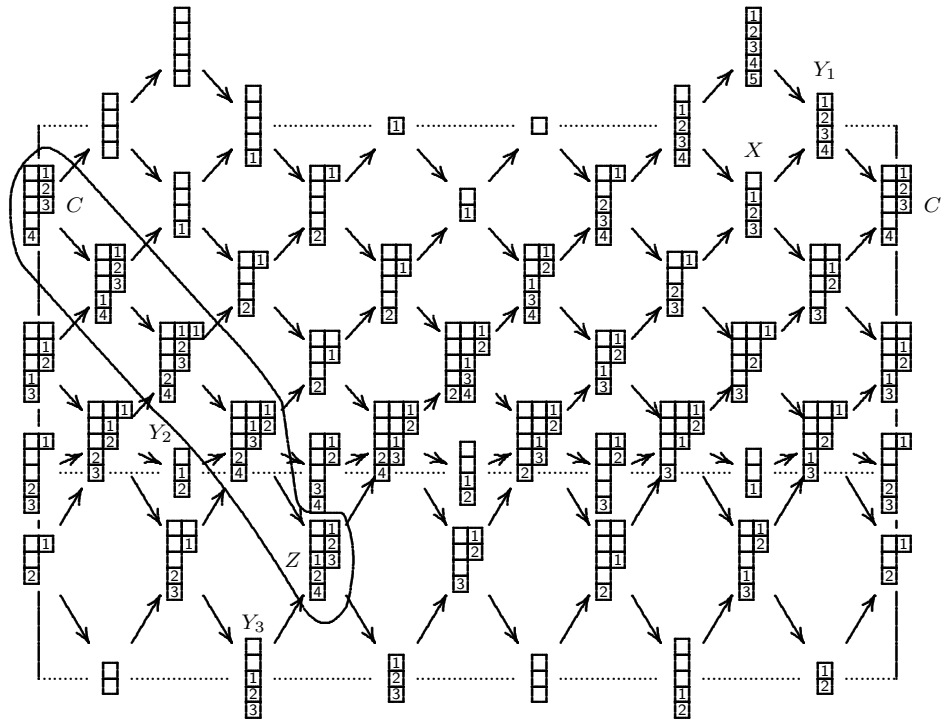
\vspace{3cm}
\begin{thebibliography}{99}

\bibitem{ars} M.~Auslander, I.~Reiten and S.~O.~Smal\o: 
        Representation theory of artin algebras,
        Cambridge Studies in Advanced Mathematics {\bf 36}, Cambridge, 1997.

\bibitem{as} M.~Auslander and S.~O.~Smal\o: 
  Almost split sequences in subcategories,
  Journal of Algebra {\bf 69} (1981), 426-454.


\bibitem{bhw}
  Beers, D.,  Hunter, R., Walker, E.:
  Finite valuated p-groups,
  Abelian Group Theory, Springer LNM {\bf 1006}  (1983), 471--507

 
\bibitem{klein} Klein, T.:
  The multiplication of Schur-functions and 
    extensions of $p$-modules,
  J.~London Math.\ Soc.\ {\bf 43} (1968), 280--284 

\bibitem{klein2} Klein, T.:
  The Hall polynomial,
  J.~Alg.\ {\bf 12} (1969), 61--78 

\bibitem{macdonald} Macdonald, I.G.:
  Symmetric functions and Hall polynomials,
  2nd edition, Oxford University Press Inc., New York (1995)

\bibitem{rs-translation} Ringel, C.M., Schmidmeier, M.: 
  The Auslander-Reiten translation in submodule categories,
  Trans.\ AMS {\bf 360} (2008), 691--716

\bibitem{rs-inv} Ringel, C.M., Schmidmeier, M.: 
  Invariant subspaces of nilpotent linear operators. I,
  J.\ reine angew.\ Math.\ {\bf 614} (2008), 1--52

\bibitem{bounded} Schmidmeier, M.:
  Bounded submodules of modules,
  J.\ pure appl.\ alg.\ {\bf 203} (2005), 45--82.

\bibitem{lr} Schmidmeier, M.:
  The entries in the LR-tableau,
  manuscript (2009), 1--13, arXive: {\tt http://arxiv.org/abs/0903.3790}.
\end{thebibliography}
\end{document}